\documentclass[11pt, a4paper]{amsart}
\usepackage{amsmath, amssymb,amsthm}
\usepackage[all]{xy}
\usepackage{faktor}
\usepackage{tabularx}
\usepackage{marvosym}

\usepackage[applemac]{inputenc}
\newtheorem{thm}{Theorem}[section]
\newtheorem{prop}[thm]{Proposition}
\newtheorem{lemma}[thm]{Lemma}
\newtheorem{cor}[thm]{Corollary}
\newtheorem{defn}[thm]{Definition}
\newtheorem{rem}[thm]{Remark}
\newtheorem*{thm-num}{Theorem} 

\newcommand{\C}{\mathbb{C}}
\newcommand{\Q}{\mathbb{Q}}

\newcommand{\Z}{\mathbb{Z}}

\newcommand{\cC}{{\mathcal{C}}}
\newcommand{\cD}{{\mathcal{D}}}

\newcommand{\cO}{\mathcal{O}}

\newcommand{\cR}{\mathcal{R}}
\newcommand{\cT}{\mathcal{T}}
\newcommand{\cW}{\mathcal{W}}
\newcommand{\gm}{\mathfrak{m}}
\newcommand{\gp}{\mathfrak{p}}

\DeclareMathOperator{\ad}{ad}
\DeclareMathOperator{\im}{im} 
\DeclareMathOperator{\Ps}{Ps} 
 
\DeclareMathOperator{\End}{End}
\DeclareMathOperator{\Frob}{Frob}
\DeclareMathOperator{\Ind}{Ind}
\DeclareMathOperator{\Gal}{Gal}
\DeclareMathOperator{\GL}{GL}
\DeclareMathOperator{\rH}{H}
\DeclareMathOperator{\rZ}{Z}
\DeclareMathOperator{\Hom}{Hom}
\DeclareMathOperator{\Ext}{Ext}
\DeclareMathOperator{\Spec}{Spec}
\DeclareMathOperator{\Sp}{Sp}
\DeclareMathOperator{\Tr}{Tr}

\title{Ramification of the Eigencurve at classical RM points}

\author{\underline{BETINA Adel}}
\address{Universitat Polit\`ecnica de Catalunya, Facultad de Matematicas y Estadistica, 08034 Barcelona}
\email{adelbetina@gmail.com}

\begin{document}

\maketitle

\begin{abstract}
J.Bella\"iche and M.Dimitrov have shown that the $p$-adic eigencurve is smooth but not etale over the weight space at $p$-regular theta series attached to a character of a real quadratic field $F$ in which  $p$ splits. We proof in this paper the existence of an isomorphism between the subring of the completed local ring of the eigencurve at these points fixed by the Atkin-Lehner involution and an universal ring representing a pseudo-deformation problem, and one gives also a precise criterion for which the ramification index is exactly $2$. 
We finish this paper by proving the smoothness of the nearly ordinary and ordinary Hecke algebras for Hilbert modular forms over $F$ at the cuspidal-overconvergent Eisenstein points which are the base change lift for $\GL(2)_{/F}$ of these theta series. 
\end{abstract}

\section{Introduction}

Let $p$ be a prime number and $\cC$ be the $p$-adic eigencurve of tame level $N$ constructed using the Hecke operators $U_p$ and $T_\ell,<\ell>$ for $\ell \nmid Np$. Recall that $\mathcal{C}$ is reduced and there exists a flat and locally finite morphism $\kappa: \cC \rightarrow \cW$, called the weight map, where $\mathcal{W}$ is the rigid space over $\Q_{p}$ representing homomorphisms $\Z_{p}^{\times}\times (\Z /N \Z)^{\times} \rightarrow \mathbb{G}_{m}$. The eingencurve $\cC$ was introduced by Coleman-Mazur when the tame level is one (see \cite{coleman-mazur}), and by Buzzard and Chenevier for any tame level (see \cite{buzzard} and \cite{chenevier-thesis} for more details).  

Let $\mathcal{H}$ be the ring $\Z[T_{l},U_{p}]_{\ell \nmid Np}$, then by construction of $\cC$ there exists a morphism $\mathcal{H} \rightarrow \mathcal{O}_{\cC}^{rig}(\cC)$ such that we can see the element of $\mathcal{H}$ as global sections of the sheaf $\mathcal{O}^{rig}_{\cC}$, bounded by $1$ on $\cC$; therefore, the canonicals application ''system of eigenvalues'' $\cC(\Q_{p})  \rightarrow \Hom(\mathcal{H},\Q_{p})$ is injective and induces a one to one correspondence between the set of normalised overconvergent modular eigenforms with Fourier coefficients in $\C_{p}$, of tame level $N$ and of weight $k \in \mathcal{W}(\C_{p})$, having nonzero $U_{p}$-eigenvalue and the set of $\C_{p}$-valued points of weight $k$ on the eigencurve $\cC$; moreover, since the image of $\mathcal{H}$ is relatively compact in $\mathcal{O}_{\cC}^{rig}(\cC)$, and $\mathcal{O}_{\cC}^{rig}(\cC)$ is reduced, then there exists a pseudo-character $T:G_{\Q,Np} \rightarrow  \mathcal{O}_{\cC}^{rig}(\cC)$ of dimension $2$, such that $T(Frob_{\ell})=T_{\ell}$. 

The weight map $\cC \rightarrow \mathcal{W}$ is etale at non-critical $p$-regular points corresponding to classical modular forms of weight $\geq 2$. This follows from the semi-simplicity of the action of the Hecke algebra and the fact that the multiplicity of the operator $U_{p}$ is exactly one (see \cite[\S7.6.2]{coleman-mazur}, \cite[\S1.4]{hida85} and \cite{kisin}); but this result failed for the classical points of weight one (For more details see \cite[\S1.1]{D-B}, \cite[\S7.1]{D-G} and \cite[\S7.4]{D-G}).

The locus of $\cC$ where $|U_{p}| = 1$ is open and closed in $\cC$ and is called the ordinary locus of $\cC$ and denoted by $\cC^{ord}$. The ordinary locus $\cC^{ord}$ is the generic fiber of the universal $p$-ordinary Hecke algebra of tame level $N$.

Let $f$ be a classical weight one point on $\cC^{ord}$ and denote by $\rho:G_{\Q} \rightarrow \GL_ {2} (\bar{\Q}_p)$ the associated Galois representation by Deligne-Serre \cite[Proposition \S4.1 ]{deligne-serre}, which has finite image.
We fix an algebraic closure $\bar{\Q}_p$ of $\Q_{p}$ and an embedding $\imath_{p}: \bar{\Q} \hookrightarrow \bar{\Q}_p$, which determines an embedding $G_{\Q_{p}}\hookrightarrow G_{\Q}$. Since the image of $\rho$ is finite then {\small $\rho_{|G_{\Q_{p}}}=\psi_{1} \oplus \psi_{2}$}, where $\psi_{i}:G_{\Q_{p}}\rightarrow \bar{\Q}^{\times}_p$ are character. Moreover, we say that $f$ is {\it regular} at $p$ if, and only if, $\psi_{1}\ne \psi_{2}$.

 Let $\cT$ be the completed local ring of $\cC$ at $f$ and $\varLambda$ be the completed local ring of $\cW$ at $\kappa(f)$. The weight map $\kappa$ induces a finite and flat homomorphism $\kappa^{\#}: \varLambda \rightarrow \cT$ of local reduced complete rings.

Let $\mathfrak{C}$ be the category of complete noetherian local $\bar{\Q}_p$-algebras, with residue field $\bar{\Q}_ p$ whose morphisms are local homomorphisms of complete noetherian local rings which induce the identity on their residue field and let ($\cR,\rho^{ord}$) the universal 2-tuple, representing  $p$-ordinary deformation of $\rho$ (see \cite[section \S2]{D-B}). Under the assumption that $\rho$ is $p$-regular, M.Dimitrov and J.Bellaïche showed in the paper \cite[\S5.1,\S6.1,\S6.2]{D-B}  the following results.

\begin{thm-num}[J.Bellaïche-M.Dimitrov \cite{D-B}]\
\begin{enumerate} 

\item There exists a deformation $\rho_{\mathcal{T}}:G_{\Q,Np} \rightarrow \GL_{2}(\mathcal{T})$ of $\rho$ which is ordinary at $p$ and the morphism $\kappa^{\#}: \varLambda \rightarrow \cT$ sends the universal deformation of $\det \rho$ to $\det\rho_{\mathcal {T}}$.
\item $\cR$ is a discrete valuation ring and the deformation $\rho_{\mathcal{T}}$ induces an isomorphism $\cR \simeq \mathcal {T}$.
\item The morphism $\kappa^{\#}:\varLambda \rightarrow \mathcal{T}$ is ramified, if and only if, $f$ has RM by real quadratic field in which $p$ splits.
\end{enumerate}
\end{thm-num}

Let $F$ be a quadratic real field, $\epsilon_{F}:G_{\Q}/G_{F}\rightarrow \{-1,1\}$ the non trivial character and $\sigma$ a generator of $\Gal(F/\Q)$, we say that $f$ has RM by $F$ if and only if $\rho \simeq \rho \otimes \epsilon_{F}$.  By \cite[\S3.1]{hida}, there exists a character $\phi: G_{F}\rightarrow {\bar{\Q}_p}^\times$ such that $\rho\simeq \Ind_{\Q}^{F}\phi$. The embedding $\iota_p$ single out a place $v$ of $F$ above $p$ and denote by $v^{\sigma}$ the other place. The hypothesis that $\rho$ is $p$-regular implies that $\phi_{|G_{F_v}} \ne \phi^{\sigma}_{|G_{F_v}}$. Moreover, $p$ splits in $F$, then $G_{F_v}=G_{\Q_p}$, $\phi_{|G_{F_v}}=\psi_1$ and $\phi^{\sigma}_{|G_{F_v}}=\psi_2$ ($\psi_2$ is an unramified character).

Since $\rho \simeq \rho \otimes \epsilon_F$ (i.e $\rho=\Ind_{\Q}^{F}\phi$), the map given by $\rho^{ord} \rightarrow \rho^{ord} \otimes \epsilon_F$ rises to an automorphism $\tau: \cR \rightarrow \cR$, denote by $\cR_{\tau=1}$, the sub-ring of $\cR$ fixed by $\tau$. 

In section \S \ref{pseudo-deformation}, we introduce a ring $\cR^{ps}$ representing a pseudo-deformation functor of the reducible Galois representation $\rho_{|G_F}$ to the object of $\mathfrak{C}$, with some local condition at $p$ (i.e ordinary at $v$) and with invariant trace by the action of $\sigma$ on $G_F$ (see Definition \S\ref{G}). Denote by $\cR^{ps}_{red}$ the reduced structure of $\cR^{ps}$.

\begin{thm} \label{main-thm1}
There exists an isomorphism $\cR_{\tau=1}\simeq \cR_{red}^{ps}$ and the reduced ring $\cR^{ps}_{red}$ is a discrete valuation ring.
\end{thm}

Let $H\subset \bar{\Q}$ be the number field fixed by $\ker(\ad \rho)$ and let $H_{\infty,v}$ (resp. $H_{\infty,v^{\sigma}}$) be the compositum of all $\Z_p$-extensions which are unramified outside $v$ (resp. $v^{\sigma}$), $H_{\infty}$ be the compositum of $H_{\infty,v}$ and $H_{\infty,v^{\sigma}}$, $L_{\infty}$ be the maximal unramified abelian $p$-extension of $H_{\infty}$ and $X_{\infty}$ be the Galois group $\Gal(L_{\infty}/H_{\infty})$. It is known that $\Gal(H_{\infty}/H)\simeq \Z_{p}^{2s}$ acts by conjugation on $X_{\infty}$ and that $X_{\infty}$ is a finitely generated torsion $\Z_p[[\Gal(H_{\infty}/H)]]$-module (see \cite[\S1.1]{Greenberg}).

\begin{thm}\label{main-thm}Let $F''$ be the maximal unramified extension of $H$ contained in $H_{\infty}$ and $L_0$ be the sub-field of $L_{\infty}$ such that $\Gal(L_0/H_{\infty})$ is the largest quotient of  $X_{\infty}$ on which $\Gal(H_{\infty}/\Q)$ acts via $\epsilon_F$, we have:
\begin{enumerate}

\item If $L_{0}$ is an abelian extension of $F''$ or $\Gal(L_0/H_\infty)$ is a finite group, then the ramification index $e$ of $\cC$ over $\cW$ at $f$ is exactly $2$.

\item Assume that the $p$-Hilbert class field of $H$ is trivial, the extension $H/\Q$ is biquadratic, then the ramification index $e$ of $\cC$ over $\cW$ at $f$ is $2$ if and only if $\Gal(L_0/H)$ is an abelian group.

\end{enumerate} 
\end{thm}

Our approach is inspired by the paper \cite{cho-vatsal} of Cho-Vatsal and build upon uses the techniques and results of the paper \cite{D-B}. More precisely, we show that the ramification index of $\cR_{\tau=1}\hookrightarrow \cR$ is two. The key observation made in section $\S3$, is that the ring $\cR_{\tau=1}$ is isomorphic to $\cR^{ps}_{red}$, and by computing its tangent space we show that it is isomorphic to $\varLambda$, from which follows that the ramification index of $\kappa$ is exactly two, since $\cR\simeq \cT$.

Let $\bar{\rho}=\Ind^{F}_\Q \bar{\phi}$ be the residual representation of $\rho$, where $\mathbb{F}_p$ is a finite field of characteristic $p$. Assume that $\phi$ is the Teichmuller lift of $\bar{\phi}$ and that $\bar{\phi}$ is an unramified character (i.e $F=\Q(\sqrt{N})$). Let $\gm$ denote the maximal ideal of the $p$-adic ordinary Hecke algebra $h_{\Q}=h_{\Q}(Np^{\infty})$ determined by the representation $\bar{\rho}$ and let $h_F=h_F(p^{\infty})$ (resp. $h^{n.ord}_F$)  be the $p$-ordinary (resp. $p$-nearly ordinary) Hecke algebra arising from cuspidal Hilbert modular forms of level $p^{\infty}$ for the field $F$. After a scalar extension we can consider that $h_{\Q}$ (resp. $h_{F}$ and $h^{n.ord}_F$) is an algebra over the ring of integers of $\bar{\Q}_p$.  

Langlands proved in \cite{L} that for any primitive holomorphic elliptic cusp forms $f_k \in S_k(\Gamma_1(N),\epsilon_F)$ of weight $k \geq 2$ and of Neben type character $\epsilon_F$, there exists a base change lift $\tilde{f}_k$ which is a Hilbert modular form for $\GL(2)_{/F}$ of weight $k$, level $1$, with a trivial Neben type character and such that $L(\tilde{f}_k,s)=L(\rho_{f_k\mid G_F},s)$, where $\rho_{f_{k}}$ is the $p$-adic Galois representation attached to $f_{k}$ (i.e $L(f_k,s)=L(\rho_{f_k},s)$). Moreover, Hida constructed in \cite[\S2]{hida98} an involution $\omega$ on $h_{\Q,\gm}$, and following the work of Langlands, Doi, Hida and Ishii in the papers \cite{L} and \cite{hida}, there exists a base-change morphism: $$\beta: h_F \rightarrow h_{\Q}$$ 

They constructed also an action of $\Delta=\Gal(F/\Q)$ on $h_F$ given by $\sigma(T_{q})= T_{q^{\sigma}}$. Let $\mathfrak{y}$ denote the inverse image of $\gm$ under this base-change map, then they conjectured under suitable assumptions that $$h_{F,\mathfrak{y}}/(\Delta-1)h_{F,\mathfrak{y}} \simeq  h_{\Q,\gm}^{\omega=1},$$ where $h^{\omega=1}_{\Q,\gm}$ is the fixed part of $h_{\Q,\gm}$ by the involution $w$ (see \cite[\S3.8]{hida}).

The restriction of $\rho$ to $G_F$ is the Galois representation associated to an overconvergent-cuspidal weight one $p$-stabilized Hilbert Eisenstein series $E_{(\phi,\phi^{\sigma})}$ of tame level $1$, and $E_{(\phi,\phi^{\sigma})}$ is associated to the height one prime ideal $\mathfrak{n}=\beta^{-1}(\gp_f)$ (resp. $\mathfrak{n}^{n.ord}$) of $h_F$ (resp. $h^{n.ord}_F$). 

Write $\mathbb{T}^{ord}$ for the completion of the localization of $h_F$ by $\mathfrak{n}$ and $\mathbb{T}^{ord}_{\Delta}$ for the reduced quotient of $\mathbb{T}^{ord}$ by the ideal of generated by elements of the form $\Delta(a)-a$.
\begin{thm}\label{base-change}
The base-change morphism $\beta$ induces an isomorphism of local rings $\beta_f: \mathbb{T}^{ord}_{\Delta} \simeq \cT_{+}$, where $\cT_{+}$ is the subring of $\cT$ fixed by $\tau$ under the identification $\cR\simeq \cT$.

\end{thm}

The Theorem \ref{base-change} allows us to use the exact same arguments that are given in the proof of Theorem \cite[B]{cho-vatsal} to deduce the following variant of the Conjecture \cite[3.8]{hida} without assuming that $(\phi/\phi^{2})_{\mid I_v} \ne 1$ as in \cite[B]{cho-vatsal}.

\begin{cor}\label{cor}  Assume that the following conditions hold for $\bar{\rho}$:  
\begin{enumerate}

\item The character $\bar{\phi}$ is everywhere unramified and $\bar{\phi}_{\mid G_{F_v}} \ne \bar{\phi}_{\mid G_{F_v}}^\sigma$. 

\item The restriction of $\bar{\rho}$ to $\Gal(\bar{\Q}/{\Q(\sqrt{(-1)^{(p-1)/2}p)}})$ is absolutely irreducible.
\end{enumerate}

Then the image of the base-change morphism $\beta: h_F \rightarrow h_{\Q,\gm}^{\omega=1}$ has a finite index.
\end{cor}

Let $\mathbb{T}^{n.ord}$ be the completion with respect to the maximal ideal of the localization of $h^{n.ord}_F$ by $\mathfrak{n}^{n.ord}$

\begin{thm}\label{q-base-change} Assume that $\phi$ is unramified everywhere and $\phi(\Frob_v)\ne \phi^{\sigma}(\Frob_v)$, then: 
\begin{enumerate}

\item The affine scheme $\Spec h^{n.ord}_F$ is smooth at the height one prime $\mathfrak{n}^{n.ord}$ corresponding to the cuspidal-overconvergent Eisenstein series $E_{(\phi,\phi^{\sigma})}$.

\item The affine scheme $\Spec h_F$ is smooth at the height one prime $\mathfrak{n}$ corresponding to the cuspidal-overconvergent Eisenstein series $E_{(\phi,\phi^{\sigma})}$, and in this case $\mathbb{T}^{ord} \simeq \mathbb{T}^{ord}_{\Delta}\simeq \cT_+$.

\end{enumerate} 

\end{thm}

Now let $\mathfrak{F}$ (resp. $\mathfrak{F}^{ord}) $ denote any nearly ordinary (resp. cuspidal ordinary of parallel weight) $p$-adic family which specializes to the cuspidal-overconvergent Eisenstein series $E_{(\phi,\phi^{\sigma})}$ in weight one. It follows from the theorem above that $\mathfrak{F}$ (resp. $\mathfrak{F}^{ord}$)  is unique up to a Galois conjugation, since there is only one component of $\Spec h^{n.ord}_F$ (resp. $\Spec h_F$) passing through the point $\mathfrak{n}^{n.ord}$ (resp, $\mathfrak{n}$), and it is a consequence of the fact that $\mathbb{T}^{n.ord}$ and $\mathbb{T}^{ord}$ are regular rings (hence integral domains). Moreover, $\mathfrak{F}^{ord}$ is the base change lift of a $p$-ordinary Hida Family passing through $f$.

Let us now explain the main ideas behind the proof of Theorem \ref{q-base-change}. First, we construct in proposition \ref{generalisation} a $p$-nearly ordinary deformation $$\rho_{\mathbb{T}^{n.ord}}:G_F \rightarrow \GL_2(\mathbb{T}^{n.ord})$$ of a reducible representation $\tilde{\rho}$ having an infinite image and with trace $\phi + \phi^{\sigma}$ (this construction is inspired by the paper \cite{B2}). 

After, we introduce a deformation problem $\cD^{n.ord}$ with some local conditions at $p$ and such that $\cD^{n.ord}$ is representable by $\cR^{n.ord}$ and which surjects to the local ring $\mathbb{T}^{n.ord}$ of dimension $3$.  The computation of the tangent space $t_\cD^{n.ord}$ of $\cD^{n.ord}$ represents an important part of the proof and shows that $t_\cD^{n.ord}$ is of dimension $3$ (see Theorem \ref{tg R^n}). Hence, the above surjection is an isomorphism of complete local regular rings of dimension $3$. 

Finally, we deduce that the dimension of the tangent space of the sub functor $\cD^{ord}$ of $\cD^{n.ord}$ of $p$-ordinary deformations of $\tilde{\rho}$ is one, and hence the $p$-ordinary quotient $\mathbb{T}^{ord}$ of $\mathbb{T}^{n.ord}$ is a discrete valuation ring.

\subsection*{Remark}\

(i) Suppose that the residual representation of $\rho$ over a finite field satisfies the assumptions of the theorems of Taylor-Wiles \cite{taylor-wiles} and \cite{A.Wiles}, $\ad \rho$ is irreducible and $p \geq 3$, then Cho-Vatsal showed under theses additional assumptions theorem \ref{main-thm1}.

(ii) H.Darmon, A.Lauder and V.Rotger give in \cite{darmon-2} a formula of the $q$-expansion of a generalised overconvergent form $f^{\dag}$ in the generalized space associated to $f$ (which is not classical). The coefficients of the generalised eigenform $f^{\dag}$ are expressed as $p$-adic logarithms
of algebraic numbers.

(iii) S.Cho provided in \cite[\S 7]{cho} many examples when the ramification index $e$ of $\cC$ over $\cW$ at $f$ is exactly $2$. More precisely, she gives examples where $h_{\Q,\gm}^{\omega=1}$ is isomorphic to the Iwasawa algebra.

(iv) Dimitrov and Ghate provided in \cite[\S7.3]{D-G} many examples implying that $\cT$ is of rank two over $\varLambda$ and so the index $e$ is also $2$ in their examples.

\subsection*{Notation}
When $L$ is a number field and $S$ the places of $L$ above $Np$, we denote by $G_{L,S}$ the Galois group of the maximal extension of $L$ unramified except at the places belonging to $S$ and at infinite \textit{places}.

Let $\cO $ be the ring of integers of a $p$-adic field and $\mathbb{F}_p$ its residual field, $\mathrm{CNL}_\cO$ be the category of complete, local, Noetherian $\cO$-algebras with residue field $\mathbb{F}_p$ and whose morphisms are the local morphism of local rings inducing the identity on their residue fields. 

For any commutative ring $A$, write $M_A$ for the free $A$-module $A\oplus A$.

For any local ring $A$, write $\gm_A$ for the maximal ideal of $A$.

Let $\varLambda_\cO$ denote the Iwasawa algebra $\cO[[T]]$.

When $W$ is a representation of $G$ and $\{G_i\}_{i \in I}$ are subgroups of $G$, we will write: 

$\rH^{i}(G,W)_{G_i}=\ker \left( \rH^{i}(G,W) \rightarrow  \underset{i\in I}\oplus \rH^{i}(G_i,W) \right)$.

\subsection*{Acknowledgement} \thanks{{ \small The author would like to thank \text{Mladen Dimitrov} for his helpful comments which enriched my work. 
I would also like to thank Victor Rotger for many mathematical discussions.

The author has received funding from the European Research Council (ERC) under the European Union's Horizon 2020 research and innovation programme (grant agreement No 682152).
}}

\section{Preliminaries and some properties of $\cR$ and $\cR_{\tau=1}$}

Let $A$ be a ring in the category $\mathfrak{C}$ and $\rho_A :G_{\Q}\rightarrow \GL_{2}(A)$ a deformation of $\rho$, we say that $\rho_A$ is ordinary at $p$ if and only if $ (\rho_A)_{|G_{\Q_p}} \simeq \left(
\begin{smallmatrix}
\psi'_{A}&*\\
0 &\psi''_{A}\end{smallmatrix}\right) $ where $\psi''_{A}$ is an unramified character lifting $\psi_2$. We consider a deformation functor $\mathcal{D}: \mathfrak{C}\rightarrow \mathrm{SETS}$, given by strict equivalence classes of deformations of $\rho$, that are ordinary at $p$. By Schlesinger’s criteria the functor $\mathcal{D}$ is representable by ($\cR,\rho^{ord}$) (see \cite[section \S2]{D-B}) and denote by $t_{\cD}$ for its tangent space.

\subsection{Some properties of $\rho^{ord}$ and The ring $\cR_{\tau=1}$ }

The projective image of $\rho$ is dihedral and contains an order $2$  element which by a slight abuse of notation we denote it by $\sigma$. Let $(e_1, e_2)$ be a basis in which $\rho_{|G_F} =\phi \oplus {\phi}^\sigma$, by rescaling this basis one can assume that $\rho(\sigma)=\left(\begin{smallmatrix}
0&1\\
1&0\end{smallmatrix}\right)$ in $\mathrm{PGL}_{2}(\bar{\Q})$

We want to exhibit a suitable basis of $M_{\cR}$, such that the diagonal entries of the realization of $\rho^{ord}$ in this basis, depends only on the trace of $\Tr \rho^{ord}$. The existence of this basis  will be crucial for the section \S\ref{pseudo-deformation}.
\begin{lemma}\label{p-ord}
Let $\gamma_{0}$ be a fixed element of $G_{F_{v}}$, which lifts $\Frob_{\mathcal{P}}$ ($\iota_p: G_{F_v}\xrightarrow{\simeq} G_{\Q_p}$) and satisfies $\phi(\gamma_{0})\ne\phi^{\sigma}(\gamma_{0})$, then there exists a basis $\mathfrak{B}_{\cR}^{ord}$ of $M_{\cR}$, such that $\rho^{ord}(\gamma_{0})=\left(
\begin{smallmatrix}
*&0\\
0&*\end{smallmatrix}\right)$ and $\rho_{|G_{F_{v}}}^{ord}=\left(
\begin{smallmatrix}
*&*\\
0 &* \end{smallmatrix}\right)$ in this basis.
\end{lemma}
\begin{proof}
Let $K$ be the field of fractions of $\cR$  ($\cR$ is DVR). Since $\cR$ is Henselian (even complete) and $\phi(\gamma) \ne \phi^{\sigma}(\gamma_0)$, there exists a basis of $M_{\cR}\otimes K$ such that $\rho^{ord}\otimes K(\gamma_{0})=\left(
\begin{smallmatrix}
*&0\\
0&*\end{smallmatrix}\right)$ and $\rho_{|G_{F_{v}}}^{ord}\otimes K=\left(
\begin{smallmatrix}
*&*\\
0 &* \end{smallmatrix}\right)$ in this basis and since $\cR$ is a discrete valuation ring, we can find a basis of $M_{\cR}$ satisfying the desired conditions.
\end{proof} 

\begin{rem} \label{ind basis} 
Since $\phi(\gamma_{0})\ne \phi^{\sigma}(\gamma_{0})$, then any other basis satisfy the same assumptions of the lemma \ref{p-ord}, are obtained by conjugating the chosen basis by a diagonal matrix. Such conjugation doesn't change $a(g), d(g)$ and the product $b(g).c(g)$, where $\rho^{ord}(g)=\left(
\begin{smallmatrix}
a(g)&b(g)\\
c(g)&d(g)\end{smallmatrix}\right)$.
\end{rem}

Let $\sigma$ be a generator of the group $\Delta=\Gal (F/\Q)$ and $\epsilon_{F}$ be a quadratic character of $\Delta$. One can see that $N\rho \otimes \epsilon_{F} N=\rho$, where $N=\left(
\begin{smallmatrix}
-1&0\\
0&1\end{smallmatrix} \right)$ in $(e_1, e_2)$.
\begin{defn}
Let $g\rightarrow \left(\begin{smallmatrix}
\tilde{a}(g)&\tilde{b}(g) \\
\tilde{c}(g)&\tilde{d}(g) \end{smallmatrix}\right)$ be the realization of $\rho^{ord}$ in a basis $\mathfrak{B}_{\cR}^{ord}$, which satisfies the assumption of lemma \ref{p-ord}. Consider the automorphism $\tilde{N}$ of $\End_{\cR}(M_{\cR})$ given by $\left(\begin{smallmatrix}
-1&0\\
0&1\end{smallmatrix}\right)$ in the basis $\mathfrak{B}_{\cR}^{ord}$, then the map $\rho^{ord} \rightarrow \tilde{N}(\rho^{ord} \otimes \epsilon_{F}) \tilde{N}$ induces an automorphism $\mathfrak{t}$ of the deformation functor $\mathcal{D}$, hence an automorphism $\tau: \cR \rightarrow \cR$, with $\tau^{2}=1$. 
\end{defn}
From the fact that $\Tr \mathfrak{t}(\rho^{ord})= \Tr(\rho^{ord} \otimes \epsilon_{F})$ and by a theorem of Nyssen \cite{nyssen} and Rouquier \cite{Rouquier}, the involution $\tau$ is independent of the choice of a basis of $M_{\cR}$, in which $\tilde{N}=\left(\begin{smallmatrix}
-1&0\\
0&1\end{smallmatrix}\right)$.

Let $A$ be a ring in the category $\mathfrak{C}$. Then a deformation $\varphi_{A}:G_{\Q,Np}\rightarrow A^{\times}$ of $\det(\rho^{ord})$ is equivalent to a continuous homomorphism $h: G_{\Q,Np} \rightarrow 1 + \mathfrak{m}_{A}$. Using the class field theory, we obtain an isomorphism $\Hom(G^{ab}_{\Q_{Np}}, 1+\mathfrak{m}_{A}) \simeq \Hom(( \mathbb{Z}/N\mathbb{Z})^{\times} \times \mathbb{Z}^{\times}_{p},1+\mathfrak{m}_{A})=\Hom(1+q\Z_p,1+\gm_A)$, where $q=p$ if $p>2$, and $q = 4$ if $p=2$.

Since $1 + \mathfrak{m}_{A}$ does not contain elements of finite order and $\varLambda \simeq \bar{\Q}_p[[1+q\Z_{p}]]$, then any deformation of $\det\rho$ to the ring $A$ is obtained via a unique morphism $\varLambda \rightarrow A$. By an abuse of notation write $\kappa^{\#}: \varLambda \rightarrow \cR$ for the morphism induced by the deformation $\det\rho^{ord}$ of $\det\rho$ (i.e we identify $\cR$ and $\mathcal{T}$).

\begin{lemma} \label{inv}\ 
\begin{enumerate}
\item The involution $\tau$ is an automorphism of $\varLambda$-algebra.
\item Let $\cR_{\tau=1}$ denote the subring of $\cR$ fixed by $\tau$, then the ring $\cR_{\tau=1}$ is an object of the category $\mathfrak{C}$ and has Krull dimension equal to one.
\item The ring $\cR_{\tau=1}$ is DVR.
\item 
Write $L$ for the field of fraction of $\cR_{\tau=1}$ and $K$ for the field of fraction of $\cR$, then $L$ is equal to the set of elements of $K$ fixed by $\tau$.
\item 
The involution $\tau :\cR \rightarrow \cR$ is not trivial and the injection $\iota:\cR_{\tau=1} \rightarrow \cR$ has degrees of ramification equal to 2.
\end{enumerate}
\end{lemma}
\begin{proof}

(i)Since $\det (\rho^{ord})=\det (\tilde{N} \rho^{ord} \otimes \epsilon_F \tilde{N})$, $\tau \circ \kappa^{\#}=\kappa^{\#}$.

(ii)Since $\kappa^{\#}:\varLambda \rightarrow \cT$ is finite morphism and $\cR \simeq \cT$, $\cR_ {\tau=1}$ is finite over $\varLambda$. The fact that $\varLambda$ is Henselian ring of dimension one (even complete) implies that $\cR_{\tau=1}$ is a finite product of local rings and his Krull dimension is one. However, the ring $\cR_{\tau=1}$ is a domain ($\cR_{\tau=1} \subset \cR$), so $\cR_{\tau=1}$ is complete local ring of dimension one. 

(iii)Since $\cR_{\tau=1}$ is local domain, Noetherian and has Krull dimension equal to one, it is sufficient to show that it is integrally closed. Let $\alpha$ be any element of the the fraction field of $\cR_{\tau=1}$ such that $\alpha$ is integral over $\cR_{\tau=1}$; write $\alpha=x/y$, where $x \in \cR_{\tau=1}$ and $y\in \cR_{\tau=1}-\{0\}$. Since $\cR_{\tau=1}$ is a subring of $\cR$, $\alpha$ is integral over $\cR$. It follows that $\alpha \in \cR$ since the ring $\cR$ is DVR. However, $\tau(\alpha)=\tau(x)/\tau(y)=x/y=\alpha$, so $\tau(\alpha)=\alpha$, hence $\alpha \in \cR_{\tau=1}$.

(iv)Let $a \in K$ and assume that $\tau(a)=a$. Since $\cR$ is a valuation ring, $a\in \cR$ or $a^{-1} \in \cR$, so $a \in \cR_{\tau=1}$ or $a^{-1} \in \cR_{\tau=1}$, hence $a\in L$.

(v)Assume that $\tau$ is trivial, then $\rho^{ord}\simeq \rho^{ord}\otimes \epsilon_{F}$. According to proposition \cite[\S3.1]{hida}, $\rho^{ord}\simeq \Ind_{\Q}^{F}\phi^{ord} \text {, where } \phi^{ord}: G_{F}\rightarrow \cR^{\times} \text { a character}$. Since $\cR \simeq \mathcal{T}$, $\rho^{ord}$ is a representation associated to a primitive Hida family containing $f$. Thus, $\rho^{ord}$ is dihedral RM, so any specialization to weight $k \geq 2$ is a classical modular form of weight $k \geq 2$ which has RM by $F$. However, as is well known, there is no RM modular forms of weight $\geq 2$, so we have a contradiction. Therefore, $\tau$ is not trivial. Since $K=L^{\tau=1}$ and $\tau^{2}=1$, $L/K$ is an extension of degrees two.
\end{proof}

Denote by $\nu_{\cR}$ the valuation of $\cR$, $H$ the splitting field of $\ad \rho$ and $w_{0}$ the place of $H$ over $p$ single out by $\iota_p$. In lemma \ref{tg ps}, we need to exhibit a generator of $\gm_{\cR_{\tau=1}}$ from the entries of the matrix of $\rho^{ord}$ and in lemma \ref{eta}, we need to estimate the lower bound of the set $\nu_{\cR}(\tilde{b}(G_{H_{w_{0}^\sigma}}))$, and that is the purpose of the following proposition.

\begin{prop}\label{tg}

Let $g\rightarrow \left(
\begin{smallmatrix}
\tilde{a}(g)&\tilde{b}(g)\\
\tilde{c}(g)&\tilde{d}(g)\end{smallmatrix}\right)$ be the realization of the universal deformation $\rho^{ord}$ in the basis $\mathfrak{B}_{\cR}^{ord}$, which lifts $(e_{1},e_{2})$, then:
\begin{enumerate}
\item
There exists two elements $g_{0},h_{0} \in G_{F}$ such that the order of both $\tilde{b}(g_{0})$ and $\tilde{c}(h_{0})$ in $\cR$ is one, and If $w^{\sigma}_{0}$ is the place of $H$ above $v^{\sigma}$ induced by the embedding $\iota_p$, the image of $G_{H_{w_0^{\sigma}}}$ by $\tilde{b}$ is contained in $\gm_\cR^2$.
\item 

One always has $\dim_{\bar{\Q}_p} \rH^{1}(F, \phi^{\sigma}/ \phi)_{G_{F_v}}=1$.
\end{enumerate}
\end{prop}
\begin{proof}

(i)Let $\left(
\begin{smallmatrix}
a&b\\
c &d\end{smallmatrix}\right)$ be an element of $\ad\rho$, by proposition \cite[\S2.3]{D-B} and the fact that $G_{\Q_{p}}=G_{F_{v}}$, the restriction of $c$  (resp. $d$) to $G_{\Q_{p}}$ (resp. to $I_{p}$) rises to the following isomorphism:

$$t_{\mathcal{D}}=\ker \left(\rH^{1}(G_{\Q}, \ad \rho)\rightarrow \rH^{1}(G_{\Q_{p}}, \phi/\phi^{\sigma}) \oplus \rH^{1}(I_{p}, \bar{\Q}_p)\right) \hspace{1cm}(1)$$

We have the following decomposition of $\ad \rho$: 
\\$\ad\rho\simeq 1 \oplus \epsilon_{F} \oplus \Ind^{F}_{\Q}(\phi/\phi^{\sigma})$, given by ${\small \left(\begin{smallmatrix}
a&b\\
c&d\end{smallmatrix}\right)=\left(
\begin{smallmatrix}
a&0\\
0&d\end{smallmatrix}\right)+ 
\left(\begin{smallmatrix}
0&b\\
c&0\end{smallmatrix}\right)}$, inducing the following decomposition: 
\\${\small \rH^{1}(G_{F}, \ad\rho)\simeq \rH^{1}(G_{F}, \phi/\phi)\oplus \rH^{1}(G_{F}, \phi/\phi^{\sigma})\oplus \rH^{1}(G_{F}, \phi^{\sigma}/\phi)\oplus \rH^{1}(G_{F}, \phi^\sigma/\phi^\sigma)}$ \hspace{0.3cm} $(2)$

given by $\left(\begin{smallmatrix} a&b\\ c&d\end{smallmatrix}\right) \rightarrow (a,b,c,d)$, where the action of $\sigma \in \Gal(F/\Q)$ exchanges $a$, $d$ and $b$, $c$ (i.e $b=c^{\sigma}$, $d=a^{\sigma}$). By applying restriction-inflation exact sequence to the isomorphism $(1)$, the relation $(2)$ and \cite[\S4.2]{D-B}, we deduce that  
$\left(\begin{smallmatrix}a & b \\ c & d \end{smallmatrix}\right)\in  t_{\cD}$ if, and only if, $a=d=0$, $b=c^\sigma$, and $c\in {\small \rH^1(F,\phi^{\sigma}/\phi)_{G_{F_v}}}$. According to \cite[Theorem \S2.2]{D-B}, dim $t_{\mathcal{D}}=1$, so $c$ is not trivial and the same for $b$, since $b=c^{\sigma}$. 

Furthermore, $c_{|G_F}=b_{|G_F}^{\sigma}$, so $b_{|G_F}\in \rH^1(F,\phi /\phi^{\sigma})_{G_{F_{v^{\sigma}}}}$. The restriction-inflation exact sequence yields $b_{|G_{H}} \in \rH^1(H,\bar{\Q}_p)^{\Gal(H/F)}_{G_{H_{w_{0}^{\sigma}}}}$.

Let $\rho_{\epsilon}$ be the deformation of $\rho$ induced by the composition of the canonical projection $\cR \twoheadrightarrow \cR/\gm^{2}_{\cR}$ and $\rho^{ord}$. Since the dimension of $t_\cD$ is one, $\cR$ is DVR (see \cite[\S6.2]{D-B}) and $\cR/\gm^{2}_{\cR}$ is isomorphic to the dual numbers, so $\rho_{\epsilon} \in \mathcal{D}(\bar{\Q}_p[\epsilon])$. 

Therefore, $\rho_{\epsilon}(g)=(1+\epsilon \rho_{1}(g))\rho(g)$, where $\rho_{1}=\left(\begin{smallmatrix}
a&b\\
c&d\end{smallmatrix}\right)$ is a generator of $t_{\mathcal{D}}$. Let $g\rightarrow \left(\begin{smallmatrix}
a'(g)&b'(g)\\
c'(g)&d'(g)\end{smallmatrix}\right)$ be the realization of $\rho_\epsilon$ by a matrix. Since $\rho_{|G_F}$ is diagonal, $b\ne 0$, $c \ne 0$ and $b_{|G_{H_{w_{0}^{\sigma}}}}=0$, then $b'\ne 0$, $c'\ne 0$ and { \small $b'_{|G_{H_{w_0^\sigma}}}=0$}, hence $\tilde{b} \ne 0$, $\tilde{c} \ne 0$ modulo $\mathfrak{m}_{\cR}^{2}$. Moreover, $G_{H}= \ker(\ad\rho)$, so { \small$\tilde{b}_{|G_{H_{w_{0}^{\sigma}}}}=0$} modulo $\gm_{\cR}^{2}$. 

(ii)It follows immediately from the isomorphism $t_{\mathcal{D}} \simeq \rH^1(F,\phi^{\sigma}/\phi)_{G_{F_v}}$ and \cite[Theorem \S2.2]{D-B} (i.e $\dim_{\bar{\Q}_p} t_{\mathcal{D}}=1$).
\end{proof}

\subsection{Criterion to extend a $G_{F}$-representation to $G_{\Q}$.}

In this subsection, we give a sufficient condition for extending a representation { \small $\rho_{K}:G_ {F} \rightarrow \GL_{2}(K)$} to all $G_{\Q}$, this condition will be crucial in the proof of theorem \ref{main-thm1}.
\begin{defn}

Let $K$ be a ring and $\rho_ {K}:G_{F}\rightarrow \GL_ {n}(K)$ be a representation. Write $\rho_{K}^{t}(g)$ for $\rho_{K}(t g t^ {-1})$, where $t$ is an element of $G_{\Q}$ non trivial on $\Delta$. 

Assume the following condition: $$(C) \hspace{0.5cm}  \text{ For each } t \in G_{\Q} \text{, there exists }r(t)\in GL_{n}(K) \text{ such that }   \rho_{K}=r(t)^{-1} \rho_{K}^{t} r(t).$$
\end{defn}
\begin{prop} \label{ext1}
Let $\rho_{K}: G_{F} \rightarrow \GL_{n}(K)$ be a representation, where $K$ is a ring. Assume that the only matrices in $M_{n}(K)$ that commute with the image of $\rho_{K}$ are the scalar matrices, and $\rho_{K}$ satisfies the condition $(C)$. Then:
\begin{enumerate}
\item 
Writing $G_\Q= G_{F} \sqcup G_F.t$ for a fixed $t \in G_\Q$, so we can choose $r$ such that the following conditions are satisfied : $\forall h\in G_F$, $r(h t)=\rho_{K}(h)r(t)$ and $r(h)=\rho_{K}(h)$.  
\item 
The function $\varrho :G_{\Q}\times G_{\Q} \rightarrow K^{\times}$ defines by $\varrho(t',t)=r(t')r(t)r^{-1}(t' t)$ is an element of $\rH^{2}(G_{\Q}, K^{\times})$ for the trivial action of $G_{\Q}$. Moverover, $\varrho$ factors through $\Delta$.
\item 
If the cohomology class of $\varrho \in \rH^2(\Delta, K^{\times})$ vanishes, then there exists a representation $r: G_\Q \longrightarrow \GL_{n} (K)$ extending $\rho_K$, and if $r'$ is another extension of $\rho_ {K} $, then $r'=r \otimes \epsilon_{F}$.
\end{enumerate}
\end{prop}

\begin{proof}
See \cite[A 1.1]{Hida}.
\end{proof}

\begin{cor} \label{ext2}\
\begin{enumerate}
\item Let $\rho_{K}: G_{F} \rightarrow \GL_{n}(K)$ be a representation, where $K$ is a field. If $\rho_{K}$ satisfies the same hypothesis of lemma \ref{ext1}, there exists a finite extension $L/K$ and a representation $\rho_{L}: G_{\Q} \longrightarrow \GL_{n} (L)$ extending $\rho_{K}$.
\item 
Let $A$ be a ring in the category $\mathfrak {C}$ and $\psi_A: G_{F} \rightarrow A^{\times}$ be a character, invariant by the action of each $\sigma$ in $G_{\Q}$. Then there exists a character $\psi'_{A} : G_{\Q} \rightarrow A^{\times}$ extending $\psi_{A}$. 
\end{enumerate}
\end{cor}

\begin{proof}

(i)There exists a functorial isomorphism $\rH^{2}(\Delta, K^{\times}) \simeq K^{\times}/(K^{\times})^{2}$. Choose an element $x \in K^{\times}$ corresponding to the cohomology class of $[\varrho]$ in $\rH^{2}(\Delta,K^{\times})$. Let $L$ be a finite extension of $K$ containing $\sqrt{x}$, then the cohomology class of $[\varrho]$ in $\rH^{2}(\Delta, L^{\times})$ vanishes. Hence, we may conclude by proposition \ref{ext1}.

(ii)The residue field of $A$ is $\bar{\Q}_p$ and it is algebraically closed. Hensel lemma's implies that the group $\rH^2(\Delta, A^\times)=A^\times/(A^\times)^2$ is trivial, so the desired result follows from proposition \ref{ext1}.
\end{proof}

\section{Pseudo-deformation and the ring $\cR^{ps}$}\label{pseudo-deformation}
\subsection{Pseudo-Character and pseudo-representation}
The first occurrence of pseudo-representation appeared in the work of Wiles (see \cite{wiles}, pp $563-564$ for details), but in his definition he requires the presence of a complex conjugation $c$ which forces the pseudo-representation to depend only its trace, but for us we will replace $c$ by $\gamma_0$ which is a fixed lift of $\Frob_\mathcal{P}$ to $G_{F_{v}}$. Under the presence of $\gamma_0$, we show in lemma \ref{Pseudo-character} that a pseudo-representation depends only on its trace. 

\begin{defn}\label{pseudo}
Let $A$ be a commutative ring and $\gamma_{0}$ be a fixed lift of $\Frob_{\mathcal{P}}$ to {\small $G_{F_{v}}$} such that $\phi(\gamma_{0}) \ne \phi^{\sigma}(\gamma_{0})$. 

Let $\tilde{a}, \tilde{d}: G_{F} \rightarrow A$, $\tilde{x}: G_{F}\times G_{F} \rightarrow A$ be a three continuous functions satisfying the following conditions: For all $g,h,t,s,w,n \in G_{F}$, we have:
\\$1)$ $\tilde{a}(s t)= \tilde{a}(s).\tilde{a}(t) + \tilde{x}(s, t)$
\\$2)$ $\tilde{d}(s t)= \tilde{d}(s).\tilde{d}(t)+\tilde{x}(t, s)$ 
\\$3)$ $\tilde{x}(s, t).\tilde{x}(w, n) = \tilde{x}(s, n).\tilde{x}(w, t )$
\\$4)$ $\tilde{x}(s t, w n ) = \tilde{a}.(s).\tilde{a}(n).\tilde{x}(t, w) +\tilde{a}(n).\tilde{d}(t).\tilde{x}(s, w )+\tilde{a}(s).\tilde{d}(w).\tilde{x}(t, n)+\tilde{d}(t).\tilde{d}(w).\tilde{x}(s, n)$
\\$5)$ $\tilde{a}(1)=\tilde{d}(1)=1$ and $\tilde{x}(h, 1)=\tilde{x}(g, 1)=0$.
\\$6)$ $\tilde{x}(\gamma_{0},g)=\tilde{x}(h,\gamma_0)=0$.

Then we say that $\pi_{A}=(\tilde{a},\tilde{d},\tilde{x})$ is a pseudo-representation (see \cite[2.2.3]{wiles} for more details). The trace and determinant of $\pi_{A}$ are the functions $\Tr(\pi_{A})(g)=\tilde{a}(g)+\tilde{d}(g)$, and $\det \pi_{A}(g)=\tilde{a}(g)\tilde{d}(g)-\tilde{x}(g,g)$.
\end{defn}

\begin{defn}
Let $\pi=(\phi, \phi^{\sigma}, 0)$ be the pseudo-representation associated to the representation $\rho_{|G_{F}}$.

Let $A$ be a ring in $\mathfrak{C}$ and $\pi_{A}=(\tilde{a}_A,\tilde{d}_A,\tilde{x}_A)$ be a continuous pseudo-representation, we say that $\pi_{A}$ is a pseudo-deformation if, and only if, $\pi_{A}$ mod $\mathfrak{m}_{A} = \pi$.
\end{defn}

Meantime, the work \cite{skinner-wiles} is a reference of pseudo-deformations.

\begin{lemma}\label{Pseudo-character}\
\begin{enumerate}
\item 
Let $A$ be a ring in $\mathfrak{C}$, and $\pi_{A}=(\tilde{a}_{A},\tilde{d}_{A},\tilde{x}_{A})$ be a pseudo-deformation, then $\pi_{A}$ depends only on $\Tr\pi_{A}$ and $\det \pi_{A}$ by the following formula:
\begin{equation}\label{formule trace}
\begin{split}
\tilde{a}_{A}(g)&= \frac{\Tr \pi_{A}(\gamma_{0} g)- \lambda_{2}\Tr \pi_{A}(g)   }{\lambda_{1}-\lambda_{2}} \\ \tilde{d}_{A}(g)&= \frac{\Tr \pi_{A}(\gamma_0 g)- \lambda_1 \Tr \pi_{A}(g)}{\lambda_2-\lambda_1}
\end{split}
\end{equation}
where $\lambda_{1}=\tilde{a}(\gamma_0)$ and $\lambda_{2}=\tilde{d}(\gamma_0)$.
\item 
If $A$ is a domain, then $\pi_{A}$ depends only on its trace (i.e $\det \pi_{A}$ depends on $\Tr \pi_{A}$).
\end{enumerate}
\end{lemma}

\begin{proof}

(i)Since $\tilde{x}(\gamma_0,\gamma_0)=0$, $\det \pi_{A}(\gamma_{0})=\tilde{a}(\gamma_0)\tilde{d}(\gamma_0)$, then $\tilde{a}(\gamma_0)$ and $\tilde{d}(\gamma_0)$ are solutions of  \begin{equation}\label{pol gamma} X^{2}-\Tr \pi_{A}(\gamma_{0})X+ \det \pi_{A}(\gamma_{0}) \end{equation} By assumption {\small $\phi(\gamma_{0}) \ne \phi^{\sigma}(\gamma_{0})$}, so Hensel lemma's implies that $\tilde{a}(\gamma_0)$ and $\tilde{d}(\gamma_0)$ are the unique solution of (\ref{pol gamma}). 

Finally, the relation (\ref{formule trace}) follows directly from the relations defining pseudo-deformation.

(ii)Let $K$ be the fraction field of $A$ and $\bar{K}$ its algebraic closure. The function $\Tr \pi_{A}: G_{F} \rightarrow K$ is pseudo-character. By theorem \cite[1.1]{taylor}, there exists a unique semi-simple Galois representation $\rho_{K}:G_{F} \rightarrow \GL_{2}(\bar{K})$ such that $\Tr \rho_{K}= \Tr \pi_{A}$ and $\det \rho_{K}=\det \pi_{A}$.
\end{proof}
\subsection{Ordinary Pseudo-deformation}
In this section we will define a sub-functor of the functor pseudo-deformation of $\pi$ and will show that it is representable by an object of $\mathfrak{C}$.

\begin{defn}\label{G}

Let $\mathfrak{G}:\mathfrak{C}\rightarrow \mathrm{Set}$ be the functor of all pseudo-deformations $\pi_A=(\tilde{a}_A,\tilde{d}_A, \tilde{x}_A)$ which satisfy the following conditions: 
\begin{enumerate}
\item 
For all $h\in G_{F_{v}}$, $h' \in G_{F}$, $\tilde{x}_{A}(h',h)=0$.
\item
$\tilde{d}_{A}(g)=1$ if $g\in I_{v}$.
\item 
$\Tr \pi_{A}(t^{-1} g t)=\Tr \pi_{A}(g)$ for each $t$ in $G_{\Q}$ and $g\in G_{F}$.  
\end{enumerate}
\end{defn}

\begin{prop} \label{rep}\

\begin{enumerate}

\item 
Let $\pi'_{\epsilon}=(a',d',x')$ be an element of $\mathfrak{G}(\bar{\Q}_p[\epsilon])$, then for any $h$ in $G_F$, $\frac{x'(h,.)}{\phi^{\sigma}(h)\phi(.)}$ (resp. $\frac{x'(.,h)}{\phi^{\sigma}(.)\phi(h)}$) is an element of $\rZ^{1}(F,\phi/\phi^\sigma)$ (resp. $\rZ^{1}(F,\phi^{\sigma}/\phi)$).

\item 
The functor $\mathfrak{G}$ is representable by $(\cR^{ps}, \pi^{ps})$. 
\item
The determinant $\det\pi^{ps}$ is invariant by the action of $\sigma$.

\end{enumerate}
\end{prop}

\begin{proof}  

(i)It follows from the defining properties of a pseudo-deformation.

(ii)The functor $\mathfrak{G}$ satisfies Schlesinger's criteria, the only non-trivial point is the finiteness of the dimension of the tangent space $t_{\mathfrak{G}}$ of $\mathfrak{G}$, and this is provided by the same argument of lemma \cite[\S2.10]{skinner-wiles} (since $\rH^{1}(F,\phi/\phi^{\sigma})$ has a finite dimension).

(iii)A direct computation shows that $$\Tr \pi^{ps}(g^{2})=(\Tr \pi^{ps}(g))^{2} - 2 \det \pi^{ps}(g),$$ so the assumption follows from the fact that $\forall t\in G_\Q, \forall g\in G_F$, $\Tr \pi^{ps}(t^{-1}gt)= \Tr \pi^{ps}(g)$.

\end{proof}

\begin{lemma}\label{str} 
There exists a natural morphism $\varLambda \longrightarrow \cR^{ps}$ induced by the deformation $\det \pi^{ps}$ of $\det \pi$.
\end{lemma}
\begin{proof} According to (iii) of lemma \ref{rep} and corollary \ref{ext2}, we can extend $\det \pi^{ps}$ into a character $\varphi: G^{ab}_{\Q,Np}\rightarrow (\cR^{ps})^{\times}$ and we choose one whose reduction modulo $\gm_{\cR^{ps}}$ is equal to $\det\rho$. Therefore, there exists a unique morphism $\varLambda \longrightarrow \cR^{ps}$ which sends the universal deformation of $\det \rho$ to $\varphi$. 

\end{proof}

\subsection{Proof of the isomorphism $\cR^{ps}_{red}\simeq \cR_{\tau=1}$}

\begin{lemma}\label{g}
Let $g \rightarrow  \left(\begin{smallmatrix}
\tilde{a}_{g}&\tilde{b}_{g}\\
\tilde{c}_{g}&\tilde{d}_{g}\end{smallmatrix}\right)$ be the realization of $\rho^{ord}$ in a basis $\mathfrak{B}_{\cR}^{ord}=\{v_{1},v_{2}\}$ (see lemma \ref{p-ord}), then:
\begin{enumerate}
\item 
The $3$-tuple $\pi_{\cR_{\tau=1}}=(\tilde{a}_{|G_{F}},\tilde{d}_{|G_{F}}, \tilde{b}_{|G_{F}}\tilde{c}_{|G_{F}})$ is a pseudo-deformation of $\pi$.
\item 
There exists a unique morphism $g:\cR^{ps} \rightarrow \cR_{\tau=1}$ inducing the pseudo-deformation $\pi_{\cR_{\tau=1}}.$
\end{enumerate}
\end{lemma}
\begin{proof}

(i)It follows from the relations defining a pseudo-representation.

(ii)Since the representation $\rho^{ord}_{|G_{F}}$ is ordinary at $G_{F_{v}}$, there exists a unique morphism
$g:\cR^{ps} \rightarrow \cR$, such that $g \circ\pi^{ps}=\pi_{\cR_{\tau=1}}$. 

Moreover, the action of $\tau$ on $\Tr \rho^{ord}$ (resp. on $\det \rho^{ord}$) is given by $\Tr \rho^{ord} \rightarrow \Tr \rho^{ord} \otimes \epsilon_{F}$ (resp. $\det \rho^{ord} \rightarrow \det \rho^{ord}\otimes \epsilon_{F}$), so $\tau$ acts trivially on $\Tr \rho^{ord}_{|G_{F}}$ (resp. on $\det \rho^{ord}$). 

Since $\cR_{\tau=1}$ is henselian (even complete), $\phi(\gamma_0)\ne \phi^{\sigma}(\gamma_0)$ and $\Tr \rho^{ord}(\gamma_{0})$, $\det \rho^{ord}(\gamma_{0})$ are elements of the ring $\cR_{\tau=1}$ ($\gamma_{0} \in G_{F_{v}}\subset G_{F}$), the eigenvalues $\lambda_{1}$ and $\lambda_{2}$ of $\rho^{ord}(\gamma_0)$ are in $\cR_{\tau=1}$. 

Hence, a direct computation shows that {\small $\tilde{a}(g)= \frac{\Tr \rho^{ord}(\gamma_0 g)- \lambda_{2}\Tr \rho^{ord}(g)   }{\lambda_{1}-\lambda_{2}}$, $\tilde{d}(g)= \frac{\Tr \rho^{ord}(\gamma_0 g)- \lambda_{1}\Tr \rho^{ord}(g)   }{\lambda_{2}-\lambda_{1}}$} and { \small $\tilde{a}(gh)=\tilde{a}(g)\tilde{a}(h)+\tilde{x}(g,h)$}. Therefore, $\tau(\tilde{a}_{|G_{F}})=\tilde{a}_{|G_{F}}$,  $\tau(\tilde{d}_{|G_{F}})=\tilde{d}_{|G_{F}}$ and $\tau( \tilde{b}_{|G_{F}}. \tilde{c}_{|G_{F}})=\tilde{b}_{|G_{F}}.\tilde{c}_{|G_{F}}$, hence $g$ factors through $\cR_{\tau=1}$.
\end{proof}

\begin{lemma}\label{tg ps}\

The morphism $g: \cR^{ps} \rightarrow \cR_{\tau=1}$ is surjective.

\end{lemma}
\begin{proof}
According to lemma \ref{tg}, there exists $g_{0},h_{0}$ in $G_{F}$ such that the order of both $\tilde{b}(g_{0})$ and $\tilde{c}(h_{0})$ in $\cR$ is one, so $\tilde{x}(g_{0},h_{0})=\tilde{b}(g_0)\tilde{c}(h_0)$ is of order $2$ in $\cR$. However, $\cR_{\tau=1}$ is DVR and the injection $\iota: \cR_{\tau=1} \hookrightarrow \cR$ is ramified with a ramification index is equal to $2$, so $\tilde{b}(g_{0}) \tilde{c}(h_{0})=\tilde{x}(g_0,h_0)$ has order one in $\cR_{\tau=1}$ and since $\cR^{ps}$ is the universal ring representing the functor $\mathfrak{G}$, $\tilde{x}(g_0,h_0)$ is contained in the image of the maximal ideal of $\cR^{ps}$ by the morphism $g$.

Let $\mathcal{B}$ be the image of the morphism $g$, so $\mathcal{B}$ is a sub-algebra of $\cR_{\tau=1}$. Denote by $\nu_{\tau}$ the discrete valuation of the ring $\cR_{\tau=1}$ and by $\mathfrak{m}_{\mathcal{B}}$ for the maximal ideal of $\mathcal{B}$. It follows from the discussion above that $\gm_{\mathcal{B}}$ contains an uniformizing element of $\cR_{\tau=1}$. Write $\mathfrak{a}$ for the ideal $\mathfrak{m}_{\mathcal{B}} \cR_{\tau=1}$, so $\mathfrak{a}=\mathfrak{m}_{\cR_{\tau=1}}$ since $\gm_{\mathcal{B}}$ contains an uniformizing element of $\cR_{\tau=1}$. 

By lemma \ref{str}, the ring $\cR^{ps}$ has structure of $\varLambda$-algebra and from the fact that $\det \pi_{\cR_{\tau=1}}=g\circ \det \pi^{ps}$, $g$ is morphism of $\varLambda$-algebra.  Since $\cR_{\tau=1}$ is a finite $\varLambda$-module, the morphism $g:\cR^{ps}\rightarrow \cR_{\tau=1}$ is finite. 

Now, apply Nakayama's lemma to the $\cR^{ps}$-module $\cR_{\tau=1}$, we find that $1$ is a generator of $\cR_{\tau=1}$ as $\cR^{ps}$-module. Hence, the morphism $g$ is surjective.

\end{proof}

\subsection*{Proof of theorem \ref{main-thm1}}

We will show that the morphism $g: \cR^{ps} \rightarrow \cR_{\tau=1}$ rises to an isomorphism $\cR^{ps}/\mathfrak{N}\simeq \cR_{\tau=1}$, where $\mathfrak{N}$ is the radical of $\cR^{ps}$. According to proposition \ref{tg ps}, the morphism $g$ is surjective. If we denote by $\mathfrak{L}$ the kernel of $g$, the statement is equivalent to $\mathfrak{L} \subset \mathfrak{N}$, which means that $\Spec \cR_{\tau=1}=\Spec \cR^{ps}$.

Let $\gp$ be a prime ideal of $\cR^{ps}$, and{ \small $\pi'':\cR^{ps} \twoheadrightarrow \cR^{ps}/\gp$} be the canonical surjection. Write $K$ for the fraction field of $\cR^{ps}/\gp$ and $\pi_{\gp}=(\tilde{a}_{\gp}, \tilde{d}_{\gp}, \tilde{x}_{\gp})$ for the pseudo-deformation obtained by the composition $\pi'' \circ \pi^{ps}$. 

If $\tilde{x}_{\gp}=0$, then {\small $\rho_{K}(g)=\left(\begin{smallmatrix}
\tilde{a}_{\gp}(g)&0\\
0&\tilde{d}_{\gp}(g)\end{smallmatrix}\right)$} is the unique semi-simple representation associated to $\pi_\gp$. 

By assumption $\Tr(\rho_{K})=\Tr(\rho_{K}^{\sigma})$, so $\tilde{a}_{\gp}^{\sigma}=\tilde{d}_{\gp}$ (since the action of $\sigma$ permute $\phi$ et $\phi^{\sigma}$ and $\phi \ne \phi^{\sigma}$). Hence, $\Ind^{F}_{\Q}\tilde{a}_{\gp}$ is a representation extending $\rho_{K}$ to $G_{\Q}$.

If there exists $g_{1},h_{1} \in G_{F}$ such that  $\tilde{x}_{\gp}(g_1,h_1) \ne 0$, proposition \cite[\S2.2.1]{wiles} implies that there exists a Galois representation {\small $\rho_{K}: g \rightarrow \left(\begin{smallmatrix}
\tilde{a}_{\gp}(g) & \tilde{x}_{\gp}(g,h_{1})/\tilde{x}_{\gp}(g_{1},h_{1})\\
\tilde{x}_{\gp}(g_{1},g)&\tilde{d}_{\gp}(g) \end{smallmatrix}\right)$} such that {\small$\Tr \rho_{K}= \Tr \pi_{\gp}$}. 

The fact that $\rho_{K}(\gamma_{0})$ is diagonal with distinct eigenvalues and {\small $\tilde{x}_{\gp}(g_1,h_1)\ne 0$} implies that $\rho_{K}$ is absolutely-irreducible. Moreover, $\Tr \pi_{A}$ is invariant by the action of $\sigma$ (i.e {\small $\Tr\pi_{A}=\Tr\rho_{K}=\Tr\rho^{\sigma}_{K}$}), so \cite[Theorem \S1]{taylor} yields an isomorphism $\rho_{K}\otimes \overline{K} \simeq\rho_{K}^{\sigma} \otimes \overline{K}$. 

Hence, there exists $r(\sigma) \in \GL_{2}(L')$, where $L'$ is finite extension of $K$ such that $r(\sigma)\rho_{K}r^{-1}(\sigma)= \rho^{\sigma}_{K}$. Thus, the representation $\rho_{K}$ satisfies the assumption of corollary \ref{ext2}. Therefore, there exists a finite extension $L/L'$ and representation $\rho_L: G_{\Q} \longrightarrow \GL_2 (L)$ extending $\rho_K$. 

Let $\mathcal {A}$ be the integral closure of $\cR^{ps}/\gp$ in $L$. Since $\cR^{ps}/\gp$ is a local Nagata ring (even complete), $\mathcal {A}$ is finite over $\cR^{ps}/\gp$, by using the same argument as in (ii) of lemma \ref{inv}, we deduce that $\mathcal{A}\in\mathfrak{C}$. Moreover, $\Tr \rho_{L}(\sigma^{2})= \Tr \rho_{L}(\sigma))^{2} - 2 \det \rho_L(\sigma)$, so $\Tr(\rho_L(G_\Q))\subset \mathcal{A}$. Thus, $\Tr\rho_L: G_ {\Q}\rightarrow \mathcal{A}$ is a pseudo-character such that its reduction modulo $\gm_{\mathcal{A}}$ is equal to $\Tr \rho_{|G_{F}}$. 

According to proposition \ref{ext1}, the restriction of $\rho$ to $G_{F}$ extends uniquely to $G_{\Q}$ since $\rho \simeq \rho\otimes \epsilon_{F}$, hence theorem \cite[1]{taylor} implies that the reduction of the pseudo-character $\Tr\rho_{L}$ modulo $\gm_{A}$ is equal to $\Tr(\rho)$. 

According to a theorem of Nyssen \cite{nyssen} and Rouquier \cite{Rouquier}, there exists a deformation $\rho_{\mathcal{A}}: G_{\Q}\rightarrow \GL_2 (\mathcal{A})$ of $\rho$, such that $\Tr \rho_{\mathcal{A}}=\Tr\rho_{L}$. In addition, we have $G_{F_{v}}=G_{\Q_{p}}$ (since $p$ splits in $F$) and by construction $(\rho_{K})_{|G_{\Q_{p}}} \simeq (\rho_{\mathcal{A}}\otimes L)_{|G_{\Q_{p}}} \simeq  \left(
\begin{smallmatrix}
\psi'_{1}&*\\
0 &\psi'_{2}\end{smallmatrix}\right)$, where $\psi'_{2}:G_{\Q_{p}}\rightarrow \mathcal{A}^\times$ is an unramified character lifting $\phi^{\sigma}_{|G_{F_{v}}}$ (i.e $\psi'_2 =(\tilde{d}_{\gp})_{|G_{\Q_{p}}}$), hence proposition \cite[\S5.1]{D-B} implies that the representation $\rho_{\mathcal{A}}$ is ordinary at $p$. 

Thus, there exists a unique morphism $h: \cR \rightarrow \mathcal{A}$ inducing $\rho_\mathcal{A}$.
$$\xymatrix{
\cR^{ps} \ar@{->>}[d]^{\pi''} \ar[r]^{g}
&\cR_{\tau=1} \ar@{^{(}->}[r]^{\iota}
&\cR \ar[ld]^{h}\\
\cR^{ps}/\gp \ar@{^{(}->}[r] & \mathcal{A}}$$

The morphisms $h\circ \iota \circ g$ and $\pi''$ induce two pseudo-deformations of $\pi$ with the same trace and determinant. Now, thanks to lemma \ref{Pseudo-character} we know that a pseudo-deformation depends only on its trace and determinant, so $h\circ \iota \circ g = \pi''$. Hence, the diagram above is commutative. Thus, we have the inclusion $\mathfrak{L}\subset \gp$ and which implies that the ideal $\mathfrak{L}$ is contained in the radical of $\cR^{ps}$.

\section{Proof of the main Theorem \ref{main-thm} }

\subsection*{Hypothesis} Assume that $L_{0}$ is an abelian extension of $F''$. \hspace{1cm}  $\left(\pmb{ \mathrm{G}}\right)$

In this section, we proof that $\cR_{\tau=1}$ is isomorphic to $\varLambda$ when $\left(\pmb{ \mathrm{G}}\right)$ holds, and it is equivalent to show that the tangent space of $\cR_{\tau=1}/{\small (\gm_{\varLambda},\gm^2_{\cR_{\tau=1}})}$ is trivial when $\left(\pmb{ \mathrm{G}}\right)$ holds.

\subsection{Tangent space of $\cR_{\tau=1}$}

Denote by $t_{\cR_{\tau=1}}$ the tangent space of $\cR_{\tau=1}$, since $\cR_{\tau=1}$ is DVR (see \S\ref{inv}), the dimension of $t_{\cR_{\tau=1}}$ is one.  

Write $t'_{\cR_{\tau=1}}$ for the sub-space of $t_{\cR_{\tau=1}}$ of pseudo-deformations with fixed determinant. It follows from theorem \ref{main-thm1} that $t'_{\cR_{\tau=1}} \hookrightarrow t_{\cR_{\tau=1}} \hookrightarrow \mathfrak{G}(\bar{\Q}_{p}[\epsilon])$. One can see that the tangent space of $\cR_{\tau=1}/{\small (\gm_{\varLambda},\gm^2_{\cR_{\tau=1}})}$ is isomorphic to $t'_{\cR_{\tau=1}}$.

In the following lemma, we introduce a representation{ \small $\rho_{\tau=1}:G_F \rightarrow \GL_2(\cR_{\tau})$} which is conjugate to $\rho^{ord}_{|G_{F}}$ by a matrix with coefficient in the field of fractions of $\cR$ and such that $\Tr \rho_{\tau=1}=\pi_{\cR_{\tau=1}}$. We introduce $\rho_{\tau=1}$ in the aim to understand the ordinariness of $\pi_{\epsilon}$ at the primes places of $H$ above $v^{\sigma}$.

\begin{lemma} \label{eta}\
\begin{enumerate}
\item 
There exists a representation $\rho_{\tau=1}:G_{F}\rightarrow \GL_{2}(\cR_{\tau=1})$ such that the pseudo-representation associated to $\rho_{\tau=1}$ is $\pi_{\cR_{\tau=1}}$.
\item 
The residual representation of $\rho_{\tau=1}$ modulo $\mathfrak{m}_{\cR_{\tau=1}}$ has the following form $\tilde{\rho}(g)=\left(
\begin{smallmatrix}
\phi & \eta\\
0&\phi^{\sigma}\end{smallmatrix}\right)$, where $\eta/\phi^{\sigma} \in \rH^{1}(F,\phi/\phi^{\sigma})_{G_{F_{v^{\sigma}}}}$.

\item
There exists a basis $(e'_{1}, e'_{2})$ of $M_{{\small \bar{\Q}_p}}$ such that $\tilde{\rho}_{|G_{F_{v^{\sigma}}}}$ splits in this basis. Moreover, $\rho_{\tau=1}$ is ordinary at $v^{\sigma}$ and the line stabilized by $G_{F_{v^{\sigma}}}$ lifts $e'_{2}$. 
\end{enumerate}
\end{lemma}
\begin{proof}

(i)According to lemma \ref{tg}, there exist $g_{0},h_{0}\in G_{F}$ such that the order of both $\tilde{b}(g_{0})$ and $\tilde{c}(h_{0})$ in $\cR$ is one. By lemma \cite[2.2.1]{wiles}, $\rho_{\tau=1}(g)=\left(\begin{smallmatrix}
\tilde{a}(g) & \tilde{x}(g,h_{0})/\tilde{x}(g_{0},h_{0}) \\ \tilde{x}(g_0, g) & \tilde{d}(g) \end{smallmatrix} \right)$ is a representation of $G_F$. Since $\tilde{b}(G_F) \subset \mathfrak{m}_{\cR}$ and the order of $\tilde{b}(g_0)$ in $\cR$ is one, the order of $\frac{\tilde{x}(g,h_{0})}{\tilde{x}(g_{0},h_{0})}=\frac{\tilde{b}(g)}{\tilde{b}(g_0)}$ in $\mathrm{Frac}(\cR)$ is non-negative. Hence, $\frac{\tilde{x}(g,h_{0})}{\tilde{x}(g_{0},h_{0})}=\frac{\tilde{b}(g)}{\tilde{b}(g_0)}$ is an element of $\cR$. 

However, $\frac{\tilde{x}(g,h_{0})}{\tilde{x}(g_{0},h_{0})}$ is invariant by $\tau$, so it belongs to $\cR_{\tau=1}$.

(ii)Since for all $g\in G_F$, $\tilde{x}(g_{0},g) \in \mathfrak{m}_{\cR_{\tau=1}}$, the residual representation of $\rho_{\tau=1}$ has the following form $g\rightarrow \left(
\begin{smallmatrix}
\phi(g) & \eta(g) \\
0 &\phi^{\sigma}(g) \end{smallmatrix}\right)$, where $\eta/\phi^\sigma \in \rH^{1}(F,\phi/\phi^\sigma )$. By  (ii) of lemma \ref{tg}, $\tilde{b}(G_{H_{w_0^\sigma}}) \subset \gm^2_\cR$, so for all $g$ in $G_{H_{w_0^\sigma}}$, $\frac{\tilde{x}(g,h_{0})}{\tilde{x}(g_0,h_0)}=\frac{\tilde{b}(g)}{\tilde{b}(g_0)} \in \gm_\cR$. 

Moreover, $\frac{\tilde{x}(g,h_{0})}{\tilde{x}(g_0,h_0)}=\frac{\tilde{b}(g)}{\tilde{b}(g_0)}$ is invariant by $\tau$, so it belongs to $\gm_{\cR_{\tau=1}}$. Hence, $\eta/\phi^{\sigma}_{|G_{H_{w_{0}^{\sigma}}}}=0$, so $\eta/\phi^\sigma_{|G_{H}} \in \rH^1(H,\bar{\Q}_p)^{\Gal(H/F)}_{G_{H_{w_0^{\sigma}}}}$.

Furthermore, the restriction inflation exact-sequence yields the following isomorphism $ \rH^1(H,\bar{\Q}_p)^{\Gal(H/F)}_{G_{H_{w_0^\sigma}}} \simeq \rH^{1}(F,\phi/\phi^\sigma)_{G_{F_{v^\sigma}}}$, hence $\eta/\phi^{\sigma} \in \rH^1(F,\phi/\phi^\sigma)_{G_{F_{v^\sigma}}}$.

(iii)Observe that $\rho_{\tau=1}$ is conjugate to $\rho^{ord}_{|G_{F}}$ by the matrix  $\left(
\begin{smallmatrix}
1/\tilde{b}(g_{0})& 0\\
0&1\end{smallmatrix}\right)$, so the representation $\rho_{\tau=1} \otimes K $ is ordinary at $v^{\sigma}$. 

In addition, the representation $\tilde{\rho}_{|G_{F_{v^{\sigma}}}}$ splits since $\eta/\phi^{\sigma}\in \rH^{1}(F,\phi/\phi^{\sigma})_{G_{F_{v^{\sigma}}}}$. Since $\cR_{\tau=1}$ contains the eigenvalues of $\rho_{\tau=1}(\sigma^{-1}\gamma_{0} \sigma)$ and $\rho_{\tau=1}\otimes L$ is ordinary at $v^{\sigma}$, we may conclude by the same argument of proposition \cite[5.1]{D-B} that $\rho_{\tau=1}$ is ordinary at $v^{\sigma}$.
\end{proof}

\begin{rem}\label{H infty}\
\begin{enumerate}
\item Let $\pi_{\epsilon}=(\tilde{a}_{\epsilon},\tilde{d}_{\epsilon},\epsilon x)$ be an element of $t_{\cR_{\tau=1}}$, Then $\left(iii\right)$ of proposition \ref{eta} implies that for any $g\in G_F$, the function $x(.,g)$ is trivial at all decomposition groups $G_{H_{w}}$, where $w\mid v^{\sigma}$ (since $\eta_{|G_{H_{w_0^\sigma}}}=0$).  

\item Let $M_{v}$ (resp. $M_{v^{\sigma}}$) be the maximal unramified outside the primes above $v $ (resp. $v^{\sigma}$) abelian pro-$p$ extension of $H$, then $H_{\infty,v}$ (resp. $H_{\infty,v^{\sigma}}$) are the fixed field by the torsion part of $\Gal(M_{v}/H)$ (resp. $\Gal(M_{v^{\sigma}}/H)$). Hence any morphism from $G_{H}$ to $\bar{\Q}_p$ unramified outside $v \left(\text{resp. } v^{\sigma}\right)$ factors through $\Gal(H_{\infty,v}/H) \left(\text{resp. } \Gal(H_{\infty,v^{\sigma}}/H)\right)$, and $x(.,*)$ is trivial when one of its components belongs to $\Gal(\bar{\Q}/H_{\infty})$.
\end{enumerate}
 \end{rem}

The purpose of the following lemmas is to explains the ordinariness of the elements of $t_{\cR_{\tau=1}}$ at all primes place of $H$ above $v^{\sigma}$ and $v$.

\begin{lemma}\label{d_g}
Let $\alpha:\cR_{\tau=1}\twoheadrightarrow \cR_{\tau=1}/\mathfrak{m}^{2}_{\cR_{\tau=1}}$ be the canonical projection, $\pi'_{\epsilon}=(a',d',x')$ be the pseudo-deformation obtained by the composition $\alpha \circ \pi_{\cR_{\tau=1}}$, and $w'$ be a place of $H$ above $v^{\sigma}$, then for any $h'$ in $I_{w'}\cap \Gal(\bar{\Q}/H_{\infty})$, $a'(h')=1$.

\end{lemma}
\begin{proof}

Denote by $\rho^{\tau}_{\epsilon}$ the representation obtained by the composition $\alpha \circ \rho_{\tau=1}$ and let  $\rho^{\tau}_{\epsilon}(g)=\left(
\begin{smallmatrix}
a'(g)& b'(g)\\ c'(g)
&d'(g)\end{smallmatrix}\right)$ be the realization of $\rho^{\tau}_{\epsilon}$ in a basis $(u_1,u_2)$ of $M_{\bar{\Q}_p[\epsilon]}$, where $b'(g)= \alpha(\tilde{x}(g,h_0)/\tilde{x}(g_0,h_0))$ and $x'(g_0,g)=c'(g)$. Take any $h$ in $I_{w^{\sigma}_{0}}\cap \Gal(\bar{\Q}/H_{\infty})$, it follows from (iii) of lemma \ref{eta} that $\tilde{\rho}_{|G_{F_{v^{\sigma}}}}=\phi \oplus \phi^{\sigma}$ in the basis $(e'_{1},e'_{2})$ of $M_{{\small \bar{\Q}_p}}$, and that $\rho^{\tau}_{\epsilon}$ is ordinary at $v^{\sigma}$ in a basis $(u_{1},v_{2})$ of $M_{{\small \bar{\Q}_p}[\epsilon]}$ lifting $(e'_{1},e'_{2})$. 

More precisely, if
 $\left(\begin{smallmatrix}
a''(h)&b''(h)\\
c''(h)&d''(h)\end{smallmatrix}\right)$ is the realization of $\rho^{\tau}_{\epsilon}$ in $(u_{1},v_{2})$, then $a''(h)=1$ and $b''(h)=0$. By remark \ref{H infty}, $c'(h)=0$, so after writing the representation $\rho_{\tau=1}$ in $(u_1,v_2)$, we obtain $a''(h)=1= a'(h)$. 

Now, if $w'$ is another place above $v^{\sigma}$, such that $g(w_0^{\sigma})=w'$ for $g\in \Gal(H/F)$, we can apply the same argument for the basis $(u_{1}, (\rho^{\tau}_{\epsilon})^{-1}(g)v_{2})$.

\end{proof}

\begin{lemma}\label{a_g}
 
Let $w$ be a place of $H$ above $v$ and $\pi'_{\epsilon}=(a',d',x')$ be an element of $t_{\cR_{\tau=1}}$, then for any $g$ in $\Gal(H_{\infty}/F)$ and $h'$ in $\Gal(\bar{\Q}/H_{\infty})$, $d'(gh'g^{-1})=d'(h')$ and $d'$ is trivial of $I_{w}\cap \Gal(\bar{\Q}/H_{\infty})$.

\end{lemma}
\begin{proof}   

(i)Write $h$ for $gh'g^{-1}$, since $x'(.,.)$ is trivial when one of these component belongs to $\Gal(\bar{\Q}/H_{\infty})$ (see remark \ref{H infty}), we have:
\begin{align}
d'(h)&=d'(gh'g^{-1})=d'(g)d'(h'g^{-1})+x'(h'g^{-1},g) \\
&=d'(g)d'(h')d'(g^{-1})+\phi(h')x'(g^{-1},g)
\end{align}
Since $d'(gg^{-1})=1=d'(g)d'(g^{-1})+x'(g^{-1},g)$, $x' \in (\epsilon)$ and $\phi(h')=\phi^{\sigma}(h')$, that $d'(h)=d'(h') (1- x'(g^{-1},g)) + \phi(h') x'(g^{-1},g) 
 =d'(h')$. 

The group $\Gal(H/F)$ acts transitively on the places of $H$ above $v$ (since $H/F$ is Galois), hence we may conclude by the discussion above and the fact that $d'_{|I_{w_0}}=1$ (i.e $\pi'_{\epsilon}$ is ordinary at $v$).

\end{proof}

\subsection{Tangent space of $\cR_{\tau=1}/\gm_{\varLambda}$ and proof of Theorem \ref{main-thm}}\

Let $\pi_{\epsilon}=(\tilde{a}_{\epsilon},\tilde{d}_{\epsilon},\tilde{x}_{\epsilon})$ be the pseudo-deformation induced by the canonical projection $\pi':\cR_{\tau=1} \twoheadrightarrow \cR_{\tau=1}/{ \small(\gm_{\varLambda},\gm^2_{\cR_{\tau=1}})}$. 

We know that $\tilde{x}_{\epsilon}$ is trivial on $\Gal(\bar{\Q}/H_{\infty})$ (see remark \ref{H infty}), so on $\Gal(\bar{\Q}/H_{\infty})$ the pseudo-deformation $\pi_{\epsilon}$ is equal  to $(\tilde{a}_{\epsilon},\tilde{d}_{\epsilon},0)$, where $\tilde{a}_{\epsilon}$, $\tilde{d}_{\epsilon}$ are characters on $\Gal(\bar{\Q}/H_{\infty})$. Let $N_{\infty}$ denote the splitting field overs $\Gal(\bar{\Q}/H_{\infty})$ of $\tilde{a}_{\epsilon} \oplus \tilde{d}_{\epsilon}$.

\begin{thm}\label{unr}

Let $\pi_{\epsilon}=(\tilde{a}_{\epsilon},\tilde{d}_{\epsilon},\tilde{x}_{\epsilon})$ be the pseudo-deformation induced by the projection $\pi':\cR_{\tau=1} \twoheadrightarrow \cR_{\tau=1}/{ \small(\gm_{\varLambda},\gm^2_{\cR_{\tau=1}})}$, then: 
\begin{enumerate}

\item $N_{\infty}$ is an unramified abelian $p$-extension of $H_{\infty}$, such that the action by conjugation of $\Gal(H_{\infty}/F)$ on $\Gal(N_{\infty}/H_{\infty})$ is trivial.
\item If the assumption $\left(\pmb{ \mathrm{G}}\right)$ holds, then the pseudo-deformation $\pi_{\epsilon}=(\tilde{a}_{\epsilon},\tilde{d}_{\epsilon},\tilde{x}_{\epsilon})$ is trivial.

\item If the Iwasawa module $\Gal(L_0/H_\infty)$ is a finite groupe, then the pseudo-deformation $\pi_{\epsilon}=(\tilde{a}_{\epsilon},\tilde{d}_{\epsilon},\tilde{x}_{\epsilon})$ is trivial.
.
\item Assume that $\left(\pmb{ \mathrm{G}}\right)$ holds or the Iwasawa module $\Gal(L_0/H_\infty)$ is a finite groupe, then the morphism $\kappa^{\#}:\varLambda \rightarrow \cR_{\tau=1}$ is an isomorphism and the ramification index $e$ of $\cC$ over $\cW$ at $f$ is exactly $2$.
\end{enumerate}
\end{thm}

\begin{proof} 
(i)Take any $g$ in $\Gal(H_{\infty}/F)$ and $h$ in $\Gal(\bar{\Q}/H_{\infty})$, since we have fixed the determinant and $\tilde{x}_{\epsilon}$ is trivial when one of these component belong to $\Gal(\bar{\Q}/H_{\infty})$, lemma \ref{a_g} implies that $\tilde{a}_{\epsilon}(ghg^{-1})= \tilde{a}_{\epsilon}(h) \text{ and } \tilde{d}_{\epsilon}(ghg^{-1})=\tilde{d}_{\epsilon}(h)$. 

Therefore, the action of $\Gal(H_{\infty}/H)$ on the character $(\tilde{a}_{\epsilon})_{|\Gal(\bar{\Q}/H_{\infty})}$, $(\tilde{d}_{\epsilon})_{|\Gal(\bar{\Q}/H_{\infty})}$ is trivial, so the action of the group $\Gal(H_{\infty}/H)$ on the group $\Gal(N_{\infty}/H_{\infty})$ is trivial. 

Moreover, lemma \ref{a_g} and \ref{d_g} imply that one of the functions $\tilde{a}_{\epsilon}$, $\tilde{d}_{\epsilon}$ are trivial at $\Gal(\bar{\Q}/H_{\infty}) \cap I_{w}$, where $w$ is any places of $H$ above $p$, and since we fixed the determinant, it follows that $\tilde{a}_{\epsilon}$ and $\tilde{d}_{\epsilon}$ are trivial on $I_{w} \cap \Gal(\bar{\Q}/H_{\infty})$, where $w$ is any place of $H$ above $p$, hence $N_{\infty}/H_{\infty}$ is unramified at all prime places of $N_{\infty}$ above $p$. 

In addition, proposition \cite[\S7.1]{D-B} implies that the image of $I_{\ell}\cap \Gal(\bar{\Q}/H_{\infty})$ by $\tilde{a}_{\epsilon}$ is finite (so trivial), where $\ell \ne p$ is a prime number. Therefore, $N_{\infty}/H_{\infty}$ is an unramified extension. 

(ii)Since the extension $N_{\infty}/H_{\infty}$ is unramified, $N_{\infty}$ is a subfield of $L_{\infty}$ and since $\Gal(H_{\infty}/H)$ acts trivially on $\Gal(N_{\infty}/H_{\infty})$, $N_{\infty}$ is contained in the subfield $L_0$ and by assumption $L_0$ is an abelian extension of $F''$, hence $N_{\infty}$ is an abelian extension of $F''$. 

It follows that $(\pi_\epsilon)_{|\Gal(\bar{\Q}/F'')}$ factors through $\Gal(N_{\infty}/F'')$ which is abelian group. Thus, $\tilde{a}_{\epsilon}(gh)=\tilde{a}_{\epsilon}(hg)$ and which implies that $\tilde{x}_{\epsilon}$ is symmetric bilinear and trivial if one of its components belongs to any inertia group $I_{w}$ ($w$ is any place of $H$ above $p$). 

The fact that $\Gal(H_{\infty}/F'')$ can be expressed as a surjective image of all inertia groups implies that the function $\tilde{x}_{\epsilon}$ is trivial on $\Gal(H_{\infty}/F'')$ and since the field $F''$ is finite extension of $H$, it follows that $\tilde{x}_{\epsilon}$ is trivial on $G_{H}$, hence lemma \ref{tg} and \ref{rep} implies that $\tilde{x}_{\epsilon}$ is trivial, and since $\cR^{ps}_{red}$ is DVR, then $\pi_\epsilon$ is trivial.

(iii) By assumption and the discussion above, $N_\infty$ is a finite extension of $H_\infty$, so $N_\infty=H_\infty$ and hence we can conclude using the argument above.

(iv)Since the tangent space of $\cR_{\tau=1}/\gm_{\varLambda}$ is trivial, $\kappa^{\#}$ is an isomorphism.
\end{proof}

\begin{prop}  Assume that the $p$-Hilbert class field of $H$ is trivial, $H$ is a biquadratic extension of $\Q$ and $\Gal(L_0/H)$ is not an abelian group, then the pseudo-deformation $\pi_{\epsilon}=(\tilde{a}_{\epsilon},\tilde{d}_{\epsilon},\tilde{x}_{\epsilon})$ is not trivial and the ramification index $e$ of $\cC$ over $\cW$ at $f$ is at least $4$.
\end{prop}

\begin{proof} Our assumptions imply that $\Gal(H_v/H)$ and $\Gal(H_{v^{\sigma}}/H)$ are isomorphic to $\Z_p$, and let $\gamma_v$, $\gamma_{v^{\sigma}}$ denote respectively their topological generator. The compositum of all $\Z_p$-extensions of $H$ is isomorphic to $\Z_p^3$ (it is the compositum of $H_\infty$ with the cyclotomic $\Z_p$-extension $\Q_\infty$ of $\Q$), hence $\Gal(L_0/H_\infty)$ is of rank one over $\Z_p$, since the extension $\Q_\infty.H_\infty/H_\infty$ is ramified at $p$ (the inertia group at $v$ inside $\Gal(H_\infty.\Q_\infty/H)$ is of rank two over $\Z_p$). 

Let $\tilde{v}$ and $\tilde{v^{\sigma}}=\sigma(\tilde{v})$ be two places of $L_0$ above $v$ and $v^{\sigma}$, so we can see $\Gal(H_v/H)$ and $\Gal(H_{v^{\sigma}}/H)$ as the inertia subgroups of $\Gal(L_0/H)$ at $\tilde{v}$ and $\tilde{v}^{\sigma}$ (since the $p$-Hilbert of $H$ is trivial).  Denote by $\beta \in \Gal(L_0/H_\infty)$ for the commutator $\gamma_{v^{\sigma}}\gamma_v \gamma_{v^{\sigma}}^{-1}\gamma_v^{-1}$. Since we have assumed that $\Gal(L_0/H)$ is not abelian, then $\beta$ is a generator of $\Gal(L_0/H_\infty)$ as a $\Z_p$-module. 

Let $\rho_v$ (resp. $\rho_{v^{\sigma}}$) be a non trivial $1$-co-cycle in $\rH^{1}(F, \phi^{\sigma}/ \phi)_{G_{F_v}}$ (resp. $\rH^{1}(F, \phi/ \phi^{\sigma})_{G_{F_{v^{\sigma}}}}$), $x_\epsilon(g,g'):=\epsilon \rho_{v^{\sigma}}(g) \rho_v(g')\phi^{\sigma}(g)\phi(g')$ (see (i) of proposition \ref{rep}), $a_\epsilon(\beta)=x_\epsilon(\gamma_{v},\gamma_{v^{\sigma}}^{-1})$, $a_\epsilon(\gamma_v)=a_\epsilon(\gamma_{v^{\sigma}})=1$, so we can extend uniquely $a_\epsilon$ to a function on $G_F$ such that $a_\epsilon(gg')-a(g)a(g')=x_\epsilon(g,g')$, and let $d(g)=a^{\sigma}(g)$ for all $g \in G_F$. By construction, $\pi_\epsilon=(a_\epsilon,d_\epsilon,x_\epsilon)$ is a non reducible pseudo-deformation in $t_\mathfrak{G}$ and with determinant equal to $\phi \phi^{\sigma}$, thus $\kappa^{\#}:\varLambda \rightarrow \cR_{\tau=1} \simeq \cT_+$ is ramified.

\end{proof}

\section{Pseudo-deformations of $\bar{\rho}$ and base-change $F/\Q$}\label{Hecke algebra} 

Let $h_{\Q}$ denote the $p$-ordinary Hecke algebra (see \cite{hida85} for more details) and let $\mathfrak{p}_f$ denote the prime ideal of height one of $h_{\Q}$ corresponding to $f$, and denote by $h_{\Q,\mathfrak{p}_f}$ the completion by the ideal maximal of the localization of $h_\Q$ by ${\mathfrak{p}_{f}}$. Let $h'_{\Q}$ denote the reduced ordinary Hecke algebra of tame level $N$ as the sub-algebra of $h_{\Q}$ generated by the Hecke operators $U_{p}$, $T_{\ell}$ and $<\ell>$ for primes $\ell$ not dividing $Np$.

\begin{prop}

There exists an isomorphism between $\cT$ and $h_{\Q,\gp_f}$.

\end{prop}

\begin{proof}

$f$ corresponds to a point $x\in \cC^{ord,0}$, where $\cC^{ord,0}$ is the cuspidal locus of the ordinary locus of $\cC^{ord}$ ($\cC^{ord,0}$ is a Zariski closed of $\cC^{ord}$) and as well known $h'_\Q$ is an integral model of $\cC^{ord,0}$ (i.e $\cC^{ord,0}=\Sp h'_\Q[1/p]$). Write $h'_{\Q,\mathfrak{p}_f}$ for the completion of the localization of $h'_{\Q}$ by $\mathfrak{p}_f \cap h_{\Q}'$. Hence, by results of \cite[\S7]{Dimitrov} and \cite[\S7.2]{D-B}, we have an isomorphism $h'_{\Q,\mathfrak{p}_f } \simeq \mathcal{T}$ and an isomorphism $h'_{\Q,\mathfrak{p}_f}\simeq h_{\Q,\gp_f}$.

\end{proof}

\subsection*{Proof of Theorem \ref{base-change}}

The representation $\rho$ associated to $f$ is dihedral, so the involution $\omega$ fixes the height one primes $\gp_f$ of $h_{\Q,\gm}$ associated to $f$. In addition, after the identification $\cR \simeq \cT$, the action of $\omega$ on $\cT$ coincides with the involution $\tau$ (see \cite[\S3]{hida} and \cite[\S2]{hida98}).

Hida constructed in \cite{hida85} a pseudo-character $\Ps_{h_{\Q}}:G_{\Q,Np} \rightarrow  h_{\Q}$, sending $\Frob_{\ell}$ to $T_\ell$ for each primes $\ell \nmid Np$. It is known, that for all primes $q \nmid Np$ of $\cO_F$, the base-change map $\beta: h_F\rightarrow h_{\Q}$ sends $T_q$ to $\Ps_{h_{\Q}}(\Frob_q)$. Let $\mathfrak{n}=\beta^{-1}(\gp_f)$, so after localization the morphism $\beta$ induces a morphism of local rings $\beta_f: \mathbb{T}^{ord} \rightarrow \cT$, and it is easy to check that the values of $\beta_f$ are in $\cT_{+}$, where $\cT_{+}$ is the subring of $\cT$ fixed by $\tau$. 

Moreover, Hida constructed in \cite{hida90} a pseudo-character $\Ps_{h_{F}}: G_F \rightarrow h_F$ of dimension two, such that $\Ps_{h_{F}}(\Frob_{q})=T_{q}$ for all primes of $q$ of $\cO_{F}$ don't divide $p$. After composition by the localization map $h_F \rightarrow \mathbb{T}^{ord}$, we get a pseudo-character of dimension two $\Ps_{\mathbb{T}^{ord}}: G_F \rightarrow \mathbb{T}^{ord}$ lifting the pseudo-character $\phi + \phi^{\sigma}$, and since $\beta(\Ps_{h_F})= (\Ps_{h_\Q})_{|G_{F}}$, then $\beta_{f}(\Ps_{\mathbb{T}^{ord}})= \Tr (\rho_{\cT})_{|G_{F}}$. 

Let $S$ be the total quotient ring of $\mathbb{T}^{ord}$( $\mathbb{T}^{ord} \subset S$), then $S= \prod \mathbb{T}^{ord}_{\gp_i}$, where $\gp_i$ are the minimal primes of $\mathbb{T}^{ord}$ and as well known each $\gp_i$ corresponds to a Hida family specializes to $\rho_{|G_{F}}$. 

Hence, a result of Wiles \cite{wiles} implies that there exists a unique semi-simple Galois representation $\rho_{S}: G_{F} \rightarrow \GL_2(S)$ ordinary at $v$ and $v^{\sigma}$, such that $\Tr (\rho_S)= \Ps_{\mathbb{T}^{ord}}$. Since $\phi(\gamma_0)\ne \phi^{\sigma}(\gamma_0)$, Hensel lemma's implies that the eigenvalues of $\rho_{S}(\gamma_0)$ are distincts, hence there exists a basis of $M_S$ such that $\rho_{S}(\gamma_0)$ is diagonal, and $(\rho_{S})_{|G_{F_{v}}}$ is upper triangular is this basis. 

Therefore, lemma \ref{Pseudo-character} implies that the coefficients of the matrix of the realization of $\rho_S$ in this basis rise to a pseudo-deformation $\pi_{\mathbb{T}^{ord}}=(a,d,bc):G_F \rightarrow \mathbb{T}^{ord}$ of $\pi$, which is ordinary at $v$.  

One can see that the action of $\Delta$ fixes $\mathfrak{n}$. Denote by $\pi_{\mathbb{T}^{ord}_{\Delta} }$ the push-forward of $\pi_{\mathbb{T}^{ord}}$ by the canonical projection $\mathbb{T}^{ord} \twoheadrightarrow \mathbb{T}^{ord}_{\Delta} $. Then the trace of $\pi_{\mathbb{T}^{ord}_{\Delta} }$ is invariant by the action of $\Delta$ and $\pi_{\mathbb{T}^{ord}_{\Delta}}$ is a point of $\mathfrak{G}(\mathbb{T}^{ord}_{\Delta})$, hence there exists a morphism $h:\cR^{ps}_{red} \rightarrow \mathbb{T}^{ord}_{\Delta} $, inducing the pseudo-deformation $\pi_{\mathbb{T}^{ord}_{\Delta} }$. 

By construction, $h(\Tr \pi^{ps}(\Frob_{q}))=T_q$ for $q\nmid p$, so the morphism $h$ is surjective since the topological generator $\{T_{q}\}_{q \nmid p}$, $U_v$ and $U_{v^{\sigma}}$ over $\varLambda$ of $\mathbb{T}^{ord}_{\Delta}$ are in the image of $h$ ($\phi_{|G_{F_{v}}}\ne \phi^{\sigma}_{|G_{F_{v}}}$ implies that $U_v,U_{v^{\sigma}} \in \im h$). 

According to theorem \ref{main-thm1}, we have the isomorphisms $\cT_{+}\simeq \cR_{\tau=1}\simeq \cR_{red}^{ps}$ and by lemma \ref{Pseudo-character}, $\cR^{ps}$ is topologically generated over $\varLambda$ by $\Tr \pi^{ps}(\Frob_q)$ for all primes $q$ of $\cO_F$. Therefore, the morphism $\beta_f: \mathbb{T}^{ord} \rightarrow \cT_{+}$ is surjective ($\beta_f$ sends $T_q$ to $\Tr \rho_{\cT}(\Frob_q)$), and since the trace of $(\rho_{\cT})_{|G_F}$ is invariant by the action of $\sigma$, $\beta_f$ factors trough $\mathbb{T}^{ord}_{\Delta} $, so the Krull dimension of $\mathbb{T}^{ord}_{\Delta} $ is $\geq 1$. In addition, the Krull dimension of $h_F$ is two, hence after localization and completion by the height one prime $\mathfrak{n}$, we deduce that $\mathbb{T}^{ord}$ is of dimension $1$, hence $\mathbb{T}^{ord}_{\Delta}$ is of dimension $1$. 

It follows from Theorem \ref{main-thm1} that the tangent space of $\cR^{ps}_{red}$ has dimension $1$ and since $\mathbb{T}^{ord}_{\Delta}$ is equidimensional of dimension $1$, the surjection $h:\cR^{ps}_{red} \twoheadrightarrow \mathbb{T}^{ord}_{\Delta}$ is an isomorphism of regular local rings of dimension $1$.$\qed$

\medskip
Let $\cO$ be the ring of integers of a $p$-adic field containing the image of $\phi$, and assume until the end of this section that the $ p$-ordinary Hecke algebra $h_{\Q,\gm}$ is define over $\cO$ (i.e after an extension of scalars $h_{\Q,\gm} $ will be an object of $\mathrm{CNL}_\cO$), and assume also that the restriction of $\bar{\rho}$ to $\Gal(\bar{\Q}/{\Q(\sqrt{(-1)^{(p-1)/2}p)}})$ is absolutely irreducible and that there exists $\gamma_0 \in G_{F_v}$ such that $\bar{\phi}(\gamma_0) \ne \bar{\phi}^{\sigma}(\gamma_0)$. Then the theorem of Taylor-Wiles \cite{A.Wiles} implies that the $p$-ordinary Hecke algebra $h_{\Q,\gm}$ is isomorphic to an universal ring $R^{ord}$ representing the $p$-ordinary minimally ramified deformations of $\bar{\rho}$ to the objects of $\mathrm{CNL}_\cO$.
\begin{defn}\ 

1)Let $A$ be an object of $\mathrm{CNL}_\cO$, $\tilde{a}, \tilde{d}: G_{F_{\{p\}}} \rightarrow A$ and $\tilde{x}: G_{F_{\{p\}}}\times G_{F_{\{p\}}} \rightarrow A$ be a continuous pseudo-representation, then we say that $\pi_{A}$ is a pseudo-deformation of $\bar{\pi}$ if, and only if, $\pi_{A}$ mod $\mathfrak{m}_{A} = \bar{\pi}$. 

2)Let $\mathfrak{G}_\cO:\mathrm{CNL}_\cO \rightarrow \mathrm{Set}$ be the functor of all pseudo-deformations $\pi_A=(\tilde{a}_A,\tilde{d}_A, \tilde{x}_A)$ of $\bar{\pi}$ which satisfy the following conditions: 
\begin{enumerate}
\item 
For all $h\in G_{F_{v}}$, $h' \in G_{F}$, $\tilde{x}_{A}(h',h)=0$.
\item
$\tilde{d}_{A}(g)=1$ if $g\in I_{v}$.
\item 
$\Tr \pi_{A}(t^{-1} g t)=\Tr \pi_{A}(g)$ for each $t$ in $G_{\Q}$ and $g\in G_{F}$.

 \end{enumerate}
\end{defn}

\begin{lemma}\label{Pseudo-character2} One always has:

\begin{enumerate}

\item Let $A$ be an object of $\mathrm{CNL}_\cO$, and $\pi_{A}=(\tilde{a}_{A},\tilde{d}_{A},\tilde{x}_{A})$ be a pseudo-deformation of $\bar{\pi}$, then $\pi_{A}$ depends only on the trace $\Tr\pi_{A}=\tilde{a}(g)+\tilde{d}(g)$ and the determinant $\det \pi_{A}=\tilde{a}(g)\tilde{d}(g)-\tilde{x}(g,g)$ as follow:

\begin{equation}
\begin{split}
\tilde{a}_{A}(g)&= \frac{\Tr \pi_{A}(\gamma_{0} g)- \lambda_{2}\Tr \pi_{A}(g)   }{\lambda_{1}-\lambda_{2}} \\ \tilde{d}_{A}(g)&= \frac{\Tr \pi_{A}(\gamma_0 g)- \lambda_1 \Tr \pi_{A}(g)}{\lambda_2-\lambda_1}
\end{split}
\end{equation}
where $\lambda_{1}=\tilde{a}(\gamma_0)$ and $\lambda_{2}=\tilde{d}(\gamma_0)$ are the unique roots of the polynomial $X^2-\Tr \pi_A(\gamma_0) X + \det \pi_A(\gamma_0)$.

\item The functor $\mathfrak{G}_\cO$ is representable by $(R^{ps}, \pi_{R^{ps}})$. 

\end{enumerate}

\end{lemma}

\begin{proof}\

i) The same proof as in lemma \ref{Pseudo-character}.

ii) The functor $\mathfrak{G}_\cO$ satisfies Schlesinger's criteria, the only non-trivial point is the finiteness of the dimension of the tangent space of $\mathfrak{G}_\cO$, and this is provided by the same argument of lemma \cite[\S2.10]{skinner-wiles} (since $\rH^{1}(G_{F_{\{p\}}},\bar{\phi}/\bar{\phi}^{\sigma})$ has a finite dimension).

\end{proof}

Hensel lemma's implies that there exists a basis of $M_{R^{ord}}$ such that the universal $p$-ordinary deformation $\rho_{R^{ord}}$ satisfies the following conditions: $$\rho_{R^{ord}}(\gamma_{0})=\left(
\begin{smallmatrix}
*&0\\
0&*\end{smallmatrix}\right) \text{, and } (\rho_{R^{ord}})_{\mid G_{\Q_p}}=\left(
\begin{smallmatrix}
*&*\\
0 &* \end{smallmatrix}\right) \text{ in this basis}.$$

Therefore, by using the exact same argument that is given in the proof of lemma \ref{g}, we obtain a morphism $\alpha : R^{ps} \rightarrow h_{\Q,\gm}$ which factors through $h_{\Q,\gm}^{\omega=1}$. 

The ring $\cR^{ps}$ is just an extension by scalar to $\bar{\Q}_p$ of the completed local ring of $\Spec R^{ps}$ at the height one prime ideal associated to the pseudo-deformation $\pi$ of $\bar{\pi}$. We deduce from lemma \ref{Pseudo-character2} that $R^{ps}$ is generated over the Iwasawa algebra $\varLambda_\cO$ by the Trace of the universal pseudo-deformation (see \cite{wiles}, p564).

Now, by using the Theorem \ref{base-change} and the exact same arguments that are given in the proof of \cite[3.10]{cho-vatsal}, we deduce the following corollary, without assuming that $(\bar{\phi}/\bar{\phi}^{\sigma})^2_{\mid I_v} \ne 1$ as in Theorem \cite[B]{cho-vatsal}.

\begin{cor}  Assume that the following conditions holds for $\bar{\rho}$:  
\begin{enumerate}

\item The character $\bar{\phi}$ is everywhere unramified. 

\item $\bar{\rho}$ is $p$-distinguished and the restriction of $\bar{\rho}$ to $\Gal(\bar{\Q}/{\Q(\sqrt{(-1)^{(p-1)/2}p)}})$ is absolutely irreducible.
\end{enumerate}

Then the image of the base-change morphism $\beta: h_F \rightarrow h_{\Q,\gm}^{\omega=1}$ has a finite index, and the image of the morphism $\alpha : R^{ps} \rightarrow h_{\Q,\gm}^{\omega=1}$ is contained in $\im \beta$ and it has also a finite index in $h_{\Q,\gm}^{\omega=1}$.

\end{cor}

\section{Deformation of a reducible Galois representation and proof of theorem \ref{q-base-change}}\

\begin{lemma}
The ring $\mathbb{T}^{n.ord}$ is equidimensional of dimension $3$.
\end{lemma}
\begin{proof}
Note that the non-zero elements of the minimal prime ideals of a reduced noetherian commutative ring are precisely the zero-divisors. It is known by the work of Hida that the nearly ordinary Hecke algebra $h_F^{n.ord}$ is a finite torsion-free module over the Iwasawa algebra of three variables $\varLambda^{n.ord}_{\cO}=\cO[[T_1,T_2,T_3]]$ (see \cite{hida90}), and since the Hecke algebra $h_F^{n.ord}$ is reduced, $h_F^{n.ord}$ is necessarily an equidimensional local ring of dimension $4$, and $\mathbb{T}^{n.ord}$ is an equidimensional local ring of dimension $3$. 
\end{proof}

Let $A$ be a ring of the category $\mathfrak{C}$ and $\rho_A :G_{F}\rightarrow \GL_{2}(A)$ a deformation of $\tilde{\rho}$, then we say that $\rho_A$ is a nearly-ordinary at $p$ if and only if, we have $$ (\rho_A)_{|G_{F_{v}}} \simeq \left(
\begin{smallmatrix}
\psi'_{v,A}&*\\
0 &\psi''_{v,A}\end{smallmatrix}\right) \text{ and } (\rho_A)_{|G_{F_{v^{\sigma}}}} \simeq \left(
\begin{smallmatrix}
\psi''_{v^{\sigma},A}&0\\
* &\psi'_{v^{\sigma},A}\end{smallmatrix}\right), $$ where $\psi''_{v,A}$ is a character lifting $\phi^{\sigma}_{\mid G_{F_v}}$ and $\psi''_{v^{\sigma},A}$ is a character lifting $\phi_{\mid G_{F_{v^\sigma}}}$. Moreover, if $\psi''_{v,A}$ and $\psi''_{v^{\sigma},A}$ are unramified, then we say that $\rho_A$ is ordinary at $p$.

\begin{defn}
Let $\mathcal{D}^{n.ord}: \mathfrak{C} \rightarrow \mathrm{SETS}$ be the functor of strict equivalence classes of deformation of $\tilde{\rho}=\left(\begin{smallmatrix}
\phi & \eta \\
0 &\phi^{\sigma} \end{smallmatrix}\right)$ which are nearly ordinary at $p$, and let $\mathcal{D}^{ord}$ be the subfunctor of $\mathcal{D}^{n.ord}$ of deformations which are ordinary at $p$.

\end{defn}
Since the representation $\tilde{\rho}$ has an infinite image and $\phi(\Frob_v) \ne \phi^{\sigma}(\Frob_v)$, Schlesinger criterions imply that $\mathcal{D}^{n.ord}$ (resp. $\mathcal{D}^{ord}$)  is representable by $(\cR^{n.ord},\rho_{\cR^{n.ord}})$ (resp. $(\cR^{ord},\rho_{\cR^{ord}})$). The determinant $\det \rho_{\cR^{ord}}$ is a deformation of the determinant $\det \pi$, so $\cR^{ord}$ is endowed naturally with a structure of $\varLambda$-algebra.
\subsection{Nearly ordinary deformation of a reducible representation}\
 
Hida constructed in \cite{hida90} a pseudo-character $\Ps_{h_{F}^{n.ord}}: G_F \rightarrow h^{n.ord}_F$ of dimension two, such that for all primes of $\ell$ of $\cO_{F}$ prime to $p$, $\Ps_{h^{n.ord}_{F}}(\Frob_{\ell})$ is the Hecke operator $T_{\ell}$. After composing $\Ps_{h_{F}^{n.ord}}$ by the localization map $h^{n,ord}_F \rightarrow \mathbb{T}^{n,ord}$, we obtain a pseudo-character of dimension two $\Ps_{\mathbb{T}^{n,ord}}: G_F \rightarrow \mathbb{T}^{n,ord}$ lifting the pseudo-character $ \Tr \tilde{\rho}=\phi \oplus \phi^{\sigma}$.

Let $Q(\mathbb{T}^{n,ord}):=\prod S'_i$ be the total quotient ring of $\mathbb{T}^{n,ord}$( $\mathbb{T}^{n,ord} \subset Q(\mathbb{T}^{n,ord})$), so $Q(\mathbb{T}^{n,ord})= \prod \mathbb{T}^{n,ord}_{\mathfrak{F}_i}$, where $\mathfrak{F}_i$ runs over the minimal primes of $\mathbb{T}^{n,ord}$, and as well known each $\mathfrak{F}_i$ corresponds to a nearly ordinary $p$-adic Family which specializes to $E_{(\phi,\phi^{\sigma})}$. 

Moreover, a result of Hida \cite{hida90} implies that there exists a unique semi-simple Galois representation $\rho_{Q(\mathbb{T}^{n,ord})}: G_{F} \rightarrow \GL_2(Q(\mathbb{T}^{n,ord}))$ satisfying $\Tr (\rho_{Q(\mathbb{T}^{n,ord})})= \Ps_{\mathbb{T}^{n,ord}}$, and which is nearly ordinary at $v$ and $v^{\sigma}$, in the sens that $(\rho_{Q(\mathbb{T}^{n,ord})})_{\mid G_{F_v}}$ (resp. $(\rho_{Q(\mathbb{T}^{n,ord})})_{\mid G_{F_v^{\sigma}}}$) is the extension a character $\psi''_{\mathbb{T}^{n,ord},v}$ (resp. $\psi''_{\mathbb{T}^{n,ord},v^{\sigma}}$) lifting $\phi^{\sigma}_{\mid G_{F_v}}$ (resp. $\phi_{\mid G_{F_{v^\sigma}}}$) by a character $\psi'_{\mathbb{T}^{n,ord},v}$ (resp. $\psi'_{\mathbb{T}^{n,ord},v^{\sigma}}$). 

The direction of the precedent extensions are uniquely determined by the fact that $U_{v}(E_{(\phi,\phi^{\sigma})})=\phi^{\sigma}(\Frob_v).E_{(\phi,\phi^{\sigma})}$ and $U_{v^{\sigma}}(E_{(\phi,\phi^{\sigma})})=\phi(\Frob_{v^{\sigma}}).E_{(\phi,\phi^{\sigma})}$ (see lemma \ref{eta}).

Let $\gamma_0' \in G_{F_{v^{\sigma}}}$ such that $\phi(\gamma'_0)\ne \phi^{\sigma}(\gamma'_0)$. Hensel lemma's implies that the eigenvalues of $\rho_{Q(\mathbb{T}^{n,ord})}(\gamma'_0)$ are distincts, and hence there exists a basis $(e''_{1},e''_{2})$ of $M_{Q(\mathbb{T}^{n,ord})}$ such that $\rho_{Q(\mathbb{T}^{n,ord})}^{ord}(\gamma'_{0})=\left(
\begin{smallmatrix}
*&0\\
0&*\end{smallmatrix}\right)$ and $(\rho_{Q(\mathbb{T}^{n,ord})})_{\mid G_{F_{v^{\sigma}}}}=\left(
\begin{smallmatrix}
*&0\\
* &* \end{smallmatrix}\right)$ in this basis.

Let $a,b,c,d$ be the coefficients of the realization of $\rho_{Q(\mathbb{T}^{n,ord})}$ by matrix in the basis $(e''_{1},e''_{2})$ of $M_{Q(\mathbb{T}^{n,ord})}$, $B$ and $C$ be the $\mathbb{T}^{n,ord}$-sub modules of $Q(\mathbb{T}^{n,ord})$ generated respectively by the coefficients $b(g)$ and $c(g')$, where $g$ and $g'$ run over the elements of $G_F$.

Let $\gm_{\mathbb{T}^{n,ord}}$ be the maximal ideal of $\mathbb{T}^{n,ord}$ and $\Ext^1_{\bar{\Q}_p[G_F]}(\phi^{\sigma},\phi)_{G_{F_{v^\sigma}}}$ be the subspace of $\Ext^1_{\bar{\Q}_p[G_F]}(\phi^{\sigma},\phi)$ given by the extensions of $\phi^{\sigma}$ by $\phi$ which are trivial at $G_{F_{v^\sigma}}$.
We have the following generalization of the proposition \cite[2]{B2}.
\begin{prop}\label{generalisation} One always has :

\begin{enumerate}
\item $\Hom_{\mathbb{T}^{n,ord}}(B,\bar{\Q}_p)$ injects $\mathbb{T}^{n,ord} $-linearly in $\Ext^1_{\bar{\Q}_p[G_F]}(\phi^{\sigma},\phi)_{G_{F_{v^\sigma}}}$.
\item $B$ is an $\mathbb{T}^{n,ord}$-module of finite type and the annihilator of $B$ is zero.
\end{enumerate}
\end{prop}

\begin{proof}\

i) Using the lemma \ref{Pseudo-character}, the coefficients $a$, $d$ and $b(g).c(g')$ can be obtained only from the trace $\Ps_{\mathbb{T}^{n,ord}}$ and the determinant $\det \rho_{Q(\mathbb{T}^{n,ord})}$. Moreover, the reduction of $\Ps_{\mathbb{T}^{n,ord}}$ is $\phi + \phi^{\sigma}$, hence $(a,d,bc)$ is a pseudo-deformation of $\pi=(\phi,\phi^{\sigma},0)$, and hence $a-\phi$, $d-\phi^{\sigma}$ and $b(g)c(g')$ are in $\gm_{\mathbb{T}^{n,ord}}$. 

Denote by $\bar{b}$ the image of $b$ in $\bar{B}=B/\gm_{\mathbb{T}^{n,ord}} B$, so we get a group homomorphism:

$$ G \rightarrow \left(
\begin{smallmatrix}
\bar{\Q}_p & \bar{B} \\ 0
& \bar{\Q}_p \end{smallmatrix}\right), \text{ given by } g \rightarrow  \left(
\begin{smallmatrix}
\phi & \bar{b}(g) \\ 0
& \phi^{\sigma} \end{smallmatrix}\right) $$ 

Since the restriction of $b$ to $G_{F_{v^{\sigma}}}$ is trivial in our basis, we obtain a morphism $$j: \Hom_{\mathbb{T}^{n,ord}}(B/\gm_{\mathbb{T}^{n,ord}} B, \bar{\Q}_p) \rightarrow \Ext^1_{\bar{\Q}_p[G_F]}(\phi^{\sigma},\phi)_{G_{F_{v^\sigma}}}$$ which associates to a morphism $f: B/\gm_{\mathbb{T}^{n,ord}} B \rightarrow \bar{\Q}_p$ the $1$-cocycle $g \rightarrow f(\bar{b}(g))$ (since $b(g)c(g') \in \gm_{\mathbb{T}^{n,ord}}$). We remark that the 
choice of the basis $(e''_{1},e''_{2})$ of $M_{Q(\mathbb{T}^{n,ord})}$ implies the $1$-cocycle $g \rightarrow f(\bar{b}(g))$ is trivial on $G_{F_{v^{\sigma}}}$.

Now, we will prove that $j$ is injective. First, a direct computation shows that $$\bar{b}(\gamma_0' g \gamma^{'-1}_0g^{-1})=\frac{\bar{b}(g)}{\phi^{\sigma}(g)} (\frac{\phi(\gamma_0')}{\phi^{\sigma}(\gamma'_0)}-1)$$ and implies that $B/\gm_{\mathbb{T}^{n,ord}} B$ is generated over $\mathbb{T}^{n,ord}$ only by the elements $\bar{b}(g)$, when $g$ runs over $G_H$ ($\gamma_0' g \gamma^{'-1}_0g^{-1} \in G_H$). Now, if $f \in \Hom_{\mathbb{T}^{n,ord}}(B/\gm_{\mathbb{T}^{n,ord}} B, \bar{\Q}_p)$ such that $f(\bar{b})$ is trivial in $\Ext^1_{\bar{\Q}_p[G_F]}(\phi^{\sigma},\phi)_{G_{F_{v^\sigma}}}$, then $f(\bar{b})$ is a $1$-co boundary and the restriction of $f$ to $G_H$ is trivial (since $H$ is the splitting field of $\phi/\phi^{\sigma}$). On the other hand, $B/\gm_{\mathbb{T}^{n,ord}} B$ is generated by $\bar{b}(g)$, when $g$ runs over $G_H$, therefore $f$ is trivial. 

ii) Since the representation $\rho_{Q(\mathbb{T}^{n,ord})}$ is semi-simple, Lemma \cite[4]{B2} implies that $B$ is a finite type $ \mathbb{T}^{n,ord}$-module.

The pseudo-character $\Ps_{h_{F}^{n.ord}}$ rises to a totally odd representation $$\rho_{h_F^{n.ord}}:G_F \rightarrow \GL_2(Q(h_{F}^{n.ord})),$$ where $Q(h_{F}^{n.ord})$ is the total fraction field of $h_{F}^{n.ord}$. We have $Q(h_{F}^{n.ord})=\prod \mathfrak{I}_i$, where $\mathfrak{I}_i$ runs over the fields given by the localization of $h_{F}^{n.ord}$ at the minimal primes of $h_F^{n.ord}$ (each $\mathfrak{I}_i$ corresponds to a nearly ordinary Hida family). There exists a basis of $M_ {Q(h_F^{n.ord})}$ in which $\rho_{h_F^{n.ord}}(c)= \left(
\begin{smallmatrix}
* & 0 \\ 0
& * \end{smallmatrix}\right)$, and let $a',b',c',d'$ be the entries of the realization of $\rho_{h_F^{n.ord}}$ by a matrix in this basis. The functions $a',d'$ and $b'c'$ depend only of the trace $\Ps_{h_{F}^{n.ord}}$ and the determinant $\det \rho_{h_F^{n.ord}}$, and the values the functions $a',d'$ and $b'c'$ are in $h_{F}^{n.ord}$.

Since the noncritical classical cuspidal Hilbert modular forms are Zariski dense on each irreducible component of $\Spec h^{n.ord}_F$, then for each field $\mathfrak{I}_i$ there exists $g_i,g'_i$ in $G_F$, such that the image by projection of $b'(g_i)c'(g'_i)$ is not trivial in $\mathfrak{I}_i$. Thus, all the representations $\rho_{S'_i}$ given by composing $\rho_{Q(\mathbb{T}^{n,ord})}$ with the projection $\rho_{Q(\mathbb{T}^{n,ord})}= \prod \mathbb{T}^{n,ord}_{\mathfrak{F}_i} \rightarrow S'_i=\mathbb{T}^{n,ord}_{\mathfrak{F}_i}$ are absolutely irreducible, so the image of $B$ is each $S'_i$ is non zero and we can conclude that the annihilator of $B$ in $\mathbb{T}^{n,ord}$ is zero.
\end{proof}

\begin{cor}\label{def n.ord} One always has:
\begin{enumerate}
\item The $\mathbb{T}^{n,ord}$-module $B$ is free of rank one and there exists an adapted basis $(e''_{1},e''_{2})$ of $M_{Q(\mathbb{T}^{n,ord})}$ such that $B$ is generated over $\mathbb{T}^{n,ord}$ by $1$. 

\item In the basis $(e''_{1},e''_{2})$, the realization $\rho_{Q(\mathbb{T}^{n,ord})}(\gamma'_0)$ is diagonal and the representation $\rho_{Q(\mathbb{T}^{n,ord})}:G_F \rightarrow \GL_2(\mathbb{T}^{n,ord})$ is a nearly ordinary deformation of $\tilde{\rho}$.

\end{enumerate}
\end{cor}
\begin{proof}\

i) Since $\Ext^1_{\bar{\Q}_p[G_F]}(\phi^{\sigma},\phi)_{G_{F_{v^\sigma}}} \simeq \rH^{1}(F,\phi/\phi^{\sigma})_{G_{F_{v^\sigma}}}$, the propositions \ref{tg} and \ref{generalisation} imply that the dimension of $ \Ext^1_{\bar{\Q}_p[G_F]}(\phi^{\sigma},\phi)_{G_{F_{v^\sigma}}}$ is one and $\dim_{\bar{\Q}_p} B \otimes \bar{\Q}_p \leq 1$.  We proved in the proposition above that $B$ is a non zero finite type $\mathbb{T}^{n,ord}$-module, then Nakayama's lemma implies that $B$ is a monogenic $\mathbb{T}^{n,ord}$-module, and the fact that the annihilator of $B$ in $\mathbb{T}^{n,ord}$ is zero implies that $B$ is a free $\mathbb{T}^{n,ord}$-module of rank one. By rescaling the basis $(e''_{1},e''_{2})$, the values of $\rho_{Q(\mathbb{T}^{n,ord})}$ will be in $\GL_2(\mathbb{T}^{n,ord})$.

ii) Since any representation isomorphic to an extension of $\phi^{\sigma}$ by $\phi$ trivial on $G_{F_{v^\sigma} }$ is necessarily isomorphic to $\tilde{\rho}$, then (i) implies that $\rho_{Q(\mathbb{T}^{n,ord})}:G_F \rightarrow \GL_2(\mathbb{T}^{n,ord})$ is a deformation of $\tilde{\rho}$ and by construction $\rho_{Q(\mathbb{T}^{n,ord})}$ is nearly ordinary at $v^{\sigma}$. 

By the work of Hida \cite{hida90}, $\rho_{Q(\mathbb{T}^{n,ord})}:G_F \rightarrow \GL_2(Q(\mathbb{T}^{n,ord}))$ is nearly ordinary at $v$. Since $\phi(\Frob_v) \ne \phi^{\sigma}(\Frob_v)$, we can use the exact same argument that is given in the proof of Proposition \cite[5.1]{D-B} to prove that $\rho_{Q(\mathbb{T}^{n,ord})}:G_F \rightarrow \GL_2(\mathbb{T}^{n,ord})$ is ordinary at $v$.

\end{proof} 

\subsection{Tangent space of $\cD^{n.ord}$}\

Let $t_{\cD^{n.ord}}$ (resp. $t_{\cD^{ord}}$) denote the tangent space of $\cD^{n.ord}$ (resp. $\cD^{ord}$). The choice of the basis $(e'_{1},e'_{2})$ of $M_{\bar{\Q}_p}$ defines in lemma \ref{eta} identifies $\End_{\bar{\Q}_p}(M_{\bar{\Q}_p})$ with $M_2(\bar{\Q}_p)$. Since $\tilde{\rho}_{\mid G_{F_{v^{\sigma}}}}$ splits completely in the basis $(e'_{1},e'_{2})$, we have the following decomposition of $\bar{\Q}_p[G_{F_{v^{\sigma}}}]$-modules 

\begin{equation}\label{morphismb}
\begin{split}
(\ad \tilde{\rho})_{|G_{F_{v^{\sigma}}}} &=  \bar{\Q}_p \oplus \phi/ \phi^{\sigma} \oplus  \phi^{\sigma} /  \phi \oplus  \bar{\Q}_p \\
 \left(\begin{smallmatrix}
a&b\\
c&d\end{smallmatrix}\right) & \longrightarrow (a,b,c,d)
\end{split}
\end{equation}

Let $\mathrm{Fil}(\ad \tilde{\rho})$ be the subspace of $\ad \tilde{\rho}$ given by the following elements
$$\mathrm{Fil}(\ad \tilde{\rho})=\{g\in \End_{\bar{\Q}_p}(M_{\bar{\Q}_p}) \mid g(e_1) \subset (e_1)  \}.$$

By composing the restriction morphism $\rH^{1}(F,\ad \tilde{\rho}) \rightarrow \rH^{1}(F_{v^{\sigma}},\ad \tilde{\rho})$ and the morphism $b^{*}:\rH^{1}(F_{v^{\sigma}},\ad \tilde{\rho}) \rightarrow \rH^{1}(F_{v^{\sigma}}, \phi/\phi^{\sigma}) $ (obtained by functoriality from (\ref{morphismb})), we obtain the natural map:

\begin{equation} \label{locp}
\rH^{1}(F,\ad \tilde{\rho}) \overset{B^{*}}{\longrightarrow} \rH^{1}(F,\phi/\phi^{\sigma})
\end{equation}

Since $\tilde{\rho}$ is reducible, $\mathrm{Fil}(\ad \tilde{\rho})$ is preserved by the action of $\ad \tilde{\rho}$ and we have a natural $G_{F}$-equivariant map given by the quotient of $\ad \tilde{\rho}$ by $\mathrm{Fil}(\ad \tilde{\rho})$:
\begin{equation} \label{locp}
\begin{split}
  \ad \tilde{\rho} & \overset{C}{\longrightarrow} \bar{\Q}_p[\phi^{\sigma}/ \phi]  \\
 \left(\begin{smallmatrix}a & b \\ c & d \end{smallmatrix}\right) &  \mapsto c   
\end{split}
\end{equation}

Let $\mathfrak{r}:\rH^1(F,\phi^{\sigma}/\phi) \rightarrow \rH^1(F_{v},\phi^{\sigma}/\phi)$ denote the natural morphism given by the restriction of the $1$-cocycles to $G_{F_v}$ and $C^{*}: \rH^1(F,\ad \tilde{\rho}) \overset{ C^{*}}{\longrightarrow}  \rH^1(F,\phi^{\sigma}/\phi)$ be the morphism obtained by functoriality from (\ref{locp}). By using a standard argument of the deformation theory, we obtain the following result.
 
\begin{lemma} \label{lemmatd} We have the following isomorphism:

  $$  t_{\cD^{n.ord}}= \ker \left( \rH^1(F,\ad \tilde{\rho}) \overset{(\mathfrak{r}\circ C^{*},B^{*})}{\longrightarrow}  (\rH^1(F_{v},\phi^{\sigma}/\phi) \oplus \rH^1(F_{v^{\sigma}}, \phi/\phi^{\sigma})) \right)$$

\end{lemma}

\medskip

Let $P=\bar{\Q}_p[\phi^{\sigma}/\phi]$ be the $\bar{\Q}_p[G_F]$-module of dimension one over $\bar{\Q}_p$ and on which $ G_F$ acts by $\phi^{\sigma}/\phi$. We remark that the quotient of $\ad \tilde{\rho}$ by $\mathrm{Fil}(\ad \tilde{\rho})$ is isomorphic to $P$ as $\bar{\Q}_p[G_F]$-module.  Thus, we have an exact sequence of $\bar{\Q}_p[G_F]$-module:

\begin{equation}\label{tgsanscond}
0\rightarrow \mathrm{Fil}(\ad \tilde{\rho})\rightarrow \ad \tilde{\rho} \rightarrow P \rightarrow 0
\end{equation}
 
Since  $\phi^{\sigma}/\phi \ne 1$, $\rH^{0}(G_{F}, P)=\{0\}$, there exists a long exacts sequence of cohomology:

\begin{equation}\label{exactadrho}
0\rightarrow \rH^{1}(F,\mathrm{Fil}(\ad \tilde{\rho})) \rightarrow \rH^{1}(F, \ad \tilde{\rho}) \rightarrow \rH^{1}(F, P) \rightarrow \rH^{2}(F,\mathrm{Fil}(\ad \tilde{\rho}))
\end{equation}

We will show that $\rH^{2}(F,\mathrm{Fil}(\ad \tilde{\rho}))$ is trivial. First, we will begin by computing the dimension of  $\rH^{1}(F,\mathrm{Fil}(\ad \tilde{\rho}))$ in order to use the Global Euler characteristic formula to deduce that $\rH^{2}(F,\mathrm{Fil}(\ad \tilde{\rho}))$ vanishes. 

Under the identification $\End_{\bar{\Q}_p}(M_{\bar{\Q}_p}) \simeq M_2(\bar{\Q}_p)$, $\mathrm{Fil}(\ad \tilde{\rho})$ is just the subspace of the upper triangular matrix of $\ad \tilde{\rho}$. Since $\tilde{\rho}$ is reducible, the subspace $$\mathrm{Fil}(\ad \tilde{\rho})_0:=\{g \in \End_{\bar{\Q}_p}(M_{\bar{\Q}_p}) \mid g(e_1)=0  \}$$  of  $\mathrm{Fil}(\ad \tilde{\rho})$ is stable by the action of $G_{F}$, and the adjoint action on this sub-space is given by $\phi/\phi^{\sigma}$. We can think $\mathrm{Fil}(\ad \tilde{\rho})_0$ as the subspace of $\mathrm{Fil}(\ad \tilde{\rho})$ given by the strict upper triangular matrix and it is isomorphic to $\bar{\Q}_p[\phi/\phi^{\sigma}]$ as $\bar{\Q}_p[G_F]$-module.

  Therefore, we obtain the following exact sequence of $\bar{\Q}_p[G_F]$-module:
$$0\rightarrow \bar{\Q}_p[\phi/\phi^{\sigma}] \rightarrow \mathrm{Fil}(\ad \tilde{\rho}) \rightarrow \bar{\Q}_p^2 \rightarrow 0$$

Hence, there exists a long exact cohomology sequence:
{\small 
\begin{equation}\label{long}
0\rightarrow H^{0}(F, \mathrm{Fil}(\ad \tilde{\rho})) \rightarrow H^{0}(F, \bar{\Q}_p^2)\overset{\delta}{\longrightarrow}\rH^{1}(F,\phi/\phi^{\sigma}) \rightarrow \rH^{1}(F, \mathrm{Fil}(\ad \tilde{\rho})) \rightarrow \rH^{1}(F, \bar{\Q}_p^2) \rightarrow \rH^{2}(F,\phi/\phi^{\sigma})
\end{equation}}

\begin{lemma}\label{obstruction}
The cohomology group $\rH^{2}(F,\phi/\phi^{\sigma})$ is trivial and one always has $\dim_{\bar{\Q}_p} \rH^{1}(F, \mathrm{Fil}(\ad \tilde{\rho}))=3$.
\end{lemma}

\begin{proof}

Global Euler characteristic formula implies that: 

\begin{equation}
\begin{split}
\dim & \rH^{0}(F, \phi/\phi^{\sigma})-\dim \rH^{1}(F, \phi/\phi^{\sigma})+\dim \rH^{2}(F, \phi/\phi^{\sigma}) \\
&=\sum_{v\mid \infty} \dim (\bar{\Q}_p)^{G_{F_{v}}} - [F:\Q] 
\end{split}
\end{equation}

Since $\phi/\phi^{\sigma}$ is a totally odd character, the relation above yields that: $$-\dim_{\bar{\Q}_p} \rH^{1}(F, \phi/\phi^{\sigma})+\dim_{\bar{\Q}_p} \rH^{2}(F, \phi/\phi^{\sigma}) = -2$$ 

It follows from the proposition \ref{tg} that $\dim_{\bar{\Q}_p} \rH^{1}(F, \phi/\phi^{\sigma})=2$, so $\rH^{2}(F, \phi/\phi^{\sigma})$ is trivial. Finally, $F$ is a real quadratic extension, so $\dim_{\bar{\Q}_p} \rH^{1}(F, \bar{\Q}_p^2)=2$, and hence the long exact sequence (\ref{long}) implies that $\dim_{\bar{\Q}_p} \rH^{1}(F, \mathrm{Fil}(\ad \tilde{\rho}))=3$.
\end{proof}

\begin{cor}\label{ex-s}\
\begin{enumerate}
\item  The cohomology group $\rH^{2}(F,\mathrm{Fil}(\ad \tilde{\rho}))$ is trivial. 
\item There exists an exact sequence $$0\rightarrow \rH^{1}(F,\mathrm{Fil}(\ad \tilde{\rho})) \rightarrow \rH^{1}(F, \ad \tilde{\rho}) \overset{C^{*}}{\longrightarrow} \rH^{1}(F, \phi^{\sigma}/\phi)\rightarrow 0$$
\end{enumerate}
 \end{cor}
\begin{proof}\
i)Global Euler characteristic formula implies that: \begin{equation}
\begin{split}
\dim_{\bar{\Q}_p} & \rH^{0}(F, \mathrm{Fil}(\ad \tilde{\rho}))-\dim_{\bar{\Q}_p} \rH^{1}(F, \mathrm{Fil}(\ad \tilde{\rho}))+\dim_{\bar{\Q}_p} \rH^{2}(F, \mathrm{Fil}(\ad \tilde{\rho}))\\
&=\sum_{v\mid \infty} \dim_{\bar{\Q}_p} ( \mathrm{Fil}(\ad \tilde{\rho}))^{G_{F_{v}}} - [F:\Q] \dim_{\bar{\Q}_p} \mathrm{Fil}(\ad \tilde{\rho})
\end{split}
\end{equation} 

Hence, the result follows immediately from the fact that $\tilde{\rho}$ is a totally odd representation and $\dim_{\bar{\Q}_p} \rH^{1}(F, \mathrm{Fil}(\ad \tilde{\rho}))=3$.

ii) Since $\rH^{2}(F,\mathrm{Fil}(\ad \tilde{\rho}))=0$, the longue exact sequence (\ref{exactadrho}) is unobstructed.

\end{proof}

\begin{thm}\label{tg R^n}
One always has $\dim_{\bar{\Q}_p} t_{\cD^{n.ord}}\leq 3$ and $\dim_{\bar{\Q}_p} t_{\cD^{ord}} \leq 1$.

\end{thm}

\begin{proof}
The proposition \ref{obstruction} and the long exact sequence (\ref{long}) yield the following exact sequence : {\small 
\begin{equation}\label{long2}
 \rH^0(F,\bar{\Q}_p^2) \overset{\delta}{ \rightarrow} \rH^{1}(F,\phi/\phi^{\sigma}) \rightarrow \rH^{1}(F, \mathrm{Fil}(\ad \tilde{\rho}))  \overset{\mathfrak{i}}\rightarrow \rH^{1}(F, \bar{\Q}_p^2) \rightarrow 0
\end{equation} }
The image of $\delta$ is of dimension one over $\bar{\Q}_p$.

Now, we will add the local conditions at $v$ and $v^{\sigma}$ to ours $1$-cocycles:

\begin{align}
\begin{array}{ccccccccc}
  \rH^{0}(F, \bar{\Q}_p^2) &\overset{\delta} {\longrightarrow} & \rH^{1}(F,\phi/\phi^{\sigma}) & \overset{i}{\longrightarrow} & \rH^{1}(F, \mathrm{Fil}(\ad \tilde{\rho})) &  \overset{\mathfrak{i}}\longrightarrow & \rH^{1}(F, \bar{\Q}_p^2) & \longrightarrow & 0\\
  & & \bigg\downarrow{\mathfrak{r}'} & & \bigg\downarrow{ \scriptsize{B^{*}}} & &  & & \\
   &  & \rH^{1}(F_{v^{\sigma}},\phi/\phi^{\sigma}) & \overset{=}{\longrightarrow} &  \rH^{1}(F_{v^{\sigma}},\phi/\phi^{\sigma}) & &  &  & \end{array}
\end{align}
Where $\mathfrak{r}'$ is the map given by restriction of the $1$-co-cycles to $G_{F_{v^{\sigma}}}$.

First, we will proof that the composition of $B^*$ with $i$ is not trivial. We proceed by absurd :

Let $\rho_1$ be a $1$-co-cycle of $\rH^{1}(F, \mathrm{Fil}(\ad \tilde{\rho}))$ lying in the image of $i$, then $\rho_1(g)=\left(\begin{smallmatrix}0 & b \\ 0 & 0 \end{smallmatrix}\right)$ such that $b \in  \rH^{1}(F,\phi/\phi^{\sigma})$. Suppose that $\left(\begin{smallmatrix}0 & b \\ 0 & 0 \end{smallmatrix}\right)$ is a non trivial $1$-co-cycle (i.e $b$ isn't a $1$-co-boundary) and belonging to $ \ker(\rH^{1}(F, \mathrm{Fil}(\ad \tilde{\rho}))\overset{B^*}{\rightarrow} \rH^{1}(F_{v^{\sigma}},\phi/\phi^{\sigma}))$, so we can modify $b$ by a $1$-coboundary such that $b= \lambda \eta/\phi^{\sigma}$, where $\lambda \in \bar{\Q}_p^{\times}$ (see lemma \ref{eta}). A direct computation shows that  the $1$-cocycle $\rho_1(g)$ is the $1$-co-boundary given by $$g \rightarrow \tilde{\rho}(g) A \tilde{\rho}(g)^{-1}- A \text{ , where } A:=\left(\begin{smallmatrix}-\lambda & 0 \\ 0 & 0 \end{smallmatrix}\right).$$ Hence we have a contradiction since we have assumed that $\rho_1$ is not a $1$-coboundary, and therefore we have $$\dim_{\bar{\Q}_p}  \ker(\rH^{1}(F, \mathrm{Fil}(\ad \tilde{\rho}))\overset{B^*}{\rightarrow} \rH^{1}(F_{v^{\sigma}},\phi/\phi^{\sigma}))=2.$$ 

Moreover, we obtain from (ii) of corollary \ref{ex-s}, the lemma \ref{lemmatd} and the discussion above the following exact sequence :
\begin{equation}\label{tg:n.ord}
0 \rightarrow  (\ker(\rH^{1}(F, \mathrm{Fil}(\ad \tilde{\rho}))\overset{B^*}{\rightarrow} \rH^{1}(F_{v^{\sigma}},\phi/\phi^{\sigma})))  \overset{i}{\rightarrow}  t_{\cD^{n.ord}}  \overset{C^*}\longrightarrow  \rH^{1}(F,\phi^{\sigma}/\phi)_{G_{F_v}}
\end{equation}
  
Finally, since the vector space $\rH^{1}(F,\phi^{\sigma}/\phi)_{\mid G_{F_v}}$ is of dimension one, one may conclude from the exact sequence (\ref{tg:n.ord}) that $\dim_{\bar{\Q}_p} t_{\cD^{n.ord}}\leq 3$.

To compute the dimension of $ t_{\cD^{ord}}$, we need to add the extra conditions of ordinariness at $p$ to $t_{\cD^{n.ord}}$ and which appears in the filtration $\mathrm{Fil}(\ad \tilde{\rho})$ as follow:

From the relation (\ref{morphismb}), we have a map of $\bar{\Q}_p[G_{F_{v^{\sigma}}}]$-modules 

\begin{equation}\label{morphisma}
\begin{split}
\ad \tilde{\rho} & \longrightarrow  \bar{\Q}_p  \\
 \left(\begin{smallmatrix}
a&b\\
c&d\end{smallmatrix}\right) & \longrightarrow a
\end{split}
\end{equation}
 and inducing by functoriality a map $$A^{*}: \rH^{1}(F,\ad \tilde{\rho}) \rightarrow \Hom(G_{F_{v^{\sigma}}},\bar{\Q}_p).$$

Hence, we have the following inclusion {\small $$t_{\cD^{ord}} \subset W=\ker \left( \rH^1(F,\ad \tilde{\rho}) \overset{(\mathfrak{r}\circ C^{*},B^{*},A^{*})}{\longrightarrow}  (\rH^1(F_{v},\phi^{\sigma}/\phi) \oplus \rH^1(F_{v^{\sigma}}, \phi/\phi^{\sigma}) \oplus \Hom(G_{F_{v^{\sigma}}},\bar{\Q}_p)) \right)$$}

Let $W_0$ denote $\ker(\rH^{1}(F, \mathrm{Fil}(\ad \tilde{\rho}))\overset{(B^*,A^*)}{\longrightarrow} \rH^{1}(F_{v^{\sigma}},\phi/\phi^{\sigma})\oplus \Hom(G_{F_{v^{\sigma}}},\bar{\Q}_p) )$, so we have the exact sequence 

\begin{equation}\label{tg:pord}
0 \rightarrow W_0  \overset{i}{\rightarrow}  W\overset{C^*}\longrightarrow  \rH^{1}(F,\phi^{\sigma}/\phi)_{G_{F_v}}
\end{equation}

Therefore, the isomorphism $\ker(\rH^{1}(F, \mathrm{Fil}(\ad \tilde{\rho}))\overset{B^*}{\rightarrow} \rH^{1}(F_{v^{\sigma}},\phi/\phi^{\sigma})) \simeq \rH^{1}(F, \bar{\Q}_p^2)$ (coming from the discussion above) implies that $W_0$ is of dimension one over $\bar{\Q}_p$ and $\dim_{\bar{\Q}_p}W \leq 2$. 

On the other hand, any $1$-co-cycle  $\rho_1 \in W_0$ satisfies the condition of ordinariness at $p$ is necessarily an homomorphism in $\rH^{1}(F, \bar{\Q}_p)$ which is unramified at $v$, so trivial (since $F$ is a real quadratic extension of $\Q$). Finally, from the discussion above and the exact sequence (\ref{tg:pord}) we obtain: $$\dim_{\bar{\Q}_p} t_{\cD^{ord}}= \dim_{\bar{\Q}_p} W-1 \leq 1.$$

\end{proof}

\subsection*{Proof of the theorem \ref{q-base-change}}\

The $p$-nearly ordinary deformation $$\rho_{Q(\mathbb{T}^{n,ord})}:G_F \rightarrow \GL_2(\mathbb{T}^{n,ord})$$ of $\tilde{\rho}$ yields a canonical morphism :

\begin{equation}\label{R=T}
\cR^{n.ord} \rightarrow \mathbb{T}^{n.ord}
\end{equation}
Let $\mathfrak{n}_1:=\mathfrak{n}^{n.ord} \cap \varLambda^{n.ord}_{\cO}$ and $\widehat{\varLambda}^{n.ord}_{(1)}$ be the completed local ring of $\Spec \varLambda^{n.ord}_{\cO}$ at $\mathfrak{n}_1$. Since $h_F^{n.ord}$ is a torsion-free $\varLambda^{n.ord}_{\cO}$-module of finite type, we obtain after localization a finite torsion-free morphism $w: \widehat{\varLambda}^{n.ord}_{(1)} \rightarrow \mathbb{T}^{n.ord}$. On the other hand, the local ring $\mathbb{T}^{ord}$ is endowed naturally with structure of $\varLambda$-algebra coming from the finite flat morphism $\varLambda_\cO \rightarrow h_F$ (see \cite{hida90}).

The ring $\cR^{n.ord}$ has a canonical structure of $\widehat{\varLambda}^{n.ord}_{(1)} $-algebra (see \cite[\S 6.2]{Bet}), and the morphism (\ref{R=T}) is a morphism of $\widehat{\varLambda}^{n.ord}_{(1)} $-algebras. 

Moreover, the ring $\cR^{ord}_{\det \rho}:=\cR^{n.ord}/\gm_{ \widehat{\varLambda}^{n.ord}_{(1)}}  \cR^{n.ord} $ represents the largest $p$-ordinary quotient of $\cR^{n.ord}$ of determinant equal to $\det \tilde{\rho}$ (see \cite[\S 6.2]{Bet}).

\begin{prop}\label{summarize} One always has:

\begin{enumerate}
\item The morphism (\ref{R=T}) yields an isomorphism of complete local regular rings $\cR^{n.ord}\simeq \mathbb{T}^{n.ord}$.
\item There exists an isomorphism between local regular rings $\cR^{ord}\simeq \mathbb{T}^{ord}$.
\item There exists an isomorphism $\cR^{ord}\simeq \cR_{\tau=1}$.
\item There exists an isomorphism $\cR^{ord}_{\det \rho} \simeq  \mathbb{T}^{ord}/\gm_\varLambda \mathbb{T}^{ord} \simeq \mathbb{T}^{n.ord}/\gm_{\widehat{\varLambda}^{n.ord}_{(1)}} \mathbb{T}^{n.ord} $.
\end{enumerate}
\end{prop}

\begin{proof}\

i)We need to show first that the morphism (\ref{R=T}) is surjective. By construction, the Hecke algebra $\mathbb{T}^{n.ord}$ is generated of $\widehat{\varLambda}^{n.ord}_{(1)}$ by the Hecke operators $T_\mathfrak{q}$ with $\mathfrak{q} \nmid p$, $U_{\gp_i}$ with $\gp_i \mid p$. The morphism (\ref{R=T}) sends the trace of $\rho_{\cR^{n.ord}}(\Frob_\mathfrak{q})$ to $T_\mathfrak{q}$ when $\mathfrak{q} \nmid p$. On the other hand, the restriction of $\rho_{\cR^{n.ord}}$ to $G_{F_{\gp_i}}$ for all primes $\gp_i \mid p$ of $F$ is an extension of the character $\psi_{i,\cR^{n.ord}}''$ by the character $\psi_{i,\cR^{n.ord}}'$, where the image of the character $\psi_{i,\cR^{n.ord}}''$ in $\mathbb{T}^{n.ord}$ is just the character $\delta_{\gp_i}$ which sends $[y, F_{\gp_{i}}]$ on the Hecke operator $T(y)$, where $[.,F_{\gp_i}]:\widehat{F_{\gp_i}^{\times}} \rightarrow G^{ab}_{F_{\gp_i}}$ is the Artin symbol. Thus, $U_{\mathfrak{p}_i}=[ \pi_{\gp_i}, F_{\gp_i}]$ in the image of the morphisme (\ref{R=T}) for some uniformizing parameter $\pi_{\gp_i}$ of the local field $F_{\gp_i}$. Hence, the morphism (\ref{R=T}) is surjective and the Krull dimension of $\cR^{n.ord}$ is at least $3$. Finally, the proposition \ref{tg R^n} implies that $\cR^{n.ord}$ is a regular ring dimension $3$, and since the Krull dimension of $\mathbb{T}^{n.ord}$ is $3$, the surjection (\ref{R=T}) is an isomorphism of regular local rings of dimension $3$.

ii) We deduce from (i) and the relation \cite[(20)]{Bet} that $\cR^{ord}\simeq \mathbb{T}^{ord}$. Moreover, the Theorem \ref{tg R^n} implies that the dimension of $\gm_{\cR^{ord}}/\gm_{\cR^{ord}}^2$ is one over $\bar{\Q}_p$ (since the Krull dimension of $\mathbb{T}^{ord}$ is equal to one). Hence $\cR^{ord}$ is a discrete valuation ring.

iii) The deformation $\rho_{\tau=1}$ of $\tilde{\rho}$ (see lemma \ref{eta}) induces by functoriality an homomorphism $\cR^{ord} \rightarrow \cR_{\tau=1}$, and since $\cR_{\tau=1}$ is generated over $\varLambda$ by the trace of $\rho_{\tau=1}$ ($\cR^{ps}_{red}\simeq \cR_{\tau=1}$), this homomorphism $\cR^{ord} \rightarrow \cR_{\tau=1}$ is surjective. Finally, both $\cR^{ord}$ and $\cR_{\tau=1}$ are DVR, then this surjection rises to an isomorphism.

iv) It follows from i), ii) and the relation \cite[\S 6.2]{Bet}.

\end{proof}   

Let $S_1^{\dag}(1, \mathrm{Id} )_{/F}$ denote the space of $p$-adic cuspidal-overconvergent Hilbert modular forms over $F$ of weight $1$, tame level $1$, of trivial Nybentypus character and with coefficients in $\bar{\Q}_p$, and let $S_1^{\dag}(1, \mathrm{Id})_{/F}[[E_{(\phi,\phi^{\sigma})}]]$ be the generalised eigenspace attached to $E_{(\phi,\phi^{\sigma})}$ inside $S_1^{\dag}(1, \mathrm{Id})_{/F}$. By construction of the $p$-ordinary Hecke algebra $h_F$, there exists an isomorphism: 

$$\Hom_{\bar{\Q}_p}(\mathbb{T}^{ord}/\gm_\varLambda \mathbb{T}^{ord},\bar{\Q}_p) \simeq S_1^{\dag}(1, \mathrm{Id})_{/F}[[E_{(\phi,\phi^{\sigma})}]].$$

We have the following consequence of the proposition \ref{summarize} and which summarizes the results of this paper.

\begin{cor} Assume that $\phi$ is unramified everywhere and $\phi(\Frob_v)\ne \phi^{\sigma}(\Frob_v)$, then the following conditions are equivalents:
\begin{enumerate}
\item $\mathbb{T}^{n.ord}$ is etale over $\widehat{\varLambda}^{n.ord}_{(1)}$. 
\item $\mathbb{T}^{ord}$ is etale over $\varLambda$.
\item $\cT_+$ is etale over $\varLambda$.
\item The ramification index $e$ of $\cC$ over $\cW$ is exactly $2$.
\item The $\bar{\Q}_p$-vector space $S_1^{\dag}(1, \mathrm{Id})_{/F}[[E_{(\phi,\phi^{\sigma})}]]$ is of dimension one and it is generated by $E_{(\phi,\phi^{\sigma})}$. \end{enumerate}
\end{cor}

\begin{rem}\

In the case where the hypothesis $\left(\pmb{ \mathrm{G}}\right)$ holds, the equivalence of the above corollary holds. 

\end{rem}

\section{examples where the ramification index $e$ of $\cC$ over $\cW$ at $f$ is $2$}\

Cho, Dimitrov and Ghate provided several examples for Hida families $\mathcal{F}$ containing a
classical RM cuspform and such that the field generated by the coefficients of $\mathcal{F}$ is a quadratic extension of the fraction field of the Iwasawa algerba $\varLambda_\cO$. Thus, we have a several numerical examples for which the ramification index $e$ of $\cC$ over $\cW$ at $f$ is $2$.

\subsection{Examples provided by Dimitrov-Ghate \cite[\S7.3]{D-G}}\

Denote by $\mathbb{T}^{\mathrm{new}}_{N,\bar{\rho}}$ the $N$-New-quotient of $h_{\Q,\gm}$ acting on the space of $\varLambda_\cO$-adic ordinary cuspforms
of tame level $N$ which are $N$-New. Dimitrov and Ghate studied in \cite[\S7.3]{D-G} the Hida families specializing to RM weight one forms, and they have many examples for which the rank of $\mathbb{T}^{\mathrm{new}}_{N,\bar{\rho}}$ over the Iwasawa algebra $\varLambda_\cO$ is two. In this case, if $\mathcal{F}$ denote a $p$-adic Hida family specializing to the classical RM form $f$, then the field generated by the coefficients of $\mathcal{F}$ is obtained by adjoining to $\mathrm{Frac}(\varLambda_\cO)$ a square-root
of an element in $\varLambda_\cO$.

Their method of computation consists in studying specializations in
weights two or more, more precisely, they showed that the $p$-adic completions of the Hecke fields of modular forms $f_k$ for
the first few weights $k$ are all quadratic extensions of $\Q_p$ (see Table $1$ and Table $2$ of \cite[\S7.3]{D-G}).

\subsection{Examples provided by Cho \cite[\S7]{cho}}\

The method of computation of S.Cho in \cite[\S7]{cho} consists in studying the unramifiedness specializations of $h_{\Q,\gm}^{\omega=1}$ of higher weight in the aim to proof that $h_{\Q,\gm}^{\omega=1} \simeq \varLambda_\cO$ in many examples.

Let $H_k$ be the Hecke algebra over $\Q$ for the space of cusp forms of weight $k$, Nybentypus character $\epsilon_F$ and level $N$, $H_k^+$ be the maximal real sub-algebra of $H_k$ and $D_{+}$ be the discriminant of $H_k^+$.

A direct computation shows that the Atkin-Lehner involution acts on $H_k$ as the complexe conjugation. Therefore, when $p \nmid  D_{+}$, the specialization of $h_{\Q,\gm}^{\omega=1}$ at the weight $k$ is unramified over $\cO$, and hence $h_{\Q,\gm}^{\omega=1} \simeq \varLambda_\cO$ by \cite[Proposition 8]{Gouvea}.

Thus, It is sufficient to find examples such that the specialization of $h_{\Q,\gm}^{\omega=1}$ at higher weight $k$ is unramified over $\cO$, and Cho checked this unramifiedness from the discriminant table from \cite[Table $1$]{hida}.

\bibliographystyle{siam}

\end{document}